\documentclass[a4paper,fleqn]{article}
\usepackage{a4wide}
\usepackage{authblk}
\providecommand{\keywords}[1]{\smallskip\noindent\textbf{\emph{Keywords:}} #1}

\usepackage{ifthen}          % provides \ifthenelse
\usepackage{graphicx}
\usepackage{tikz}
\usetikzlibrary{fit,patterns,decorations.pathmorphing,decorations.pathreplacing,calc}
\usetikzlibrary{backgrounds, positioning}

\usepackage{amsmath,amssymb,amsthm}
\usepackage[inline]{enumitem}

\usepackage{thmtools}
\usepackage{thm-restate} % to restate theorems and other environments

\usepackage{etoolbox}
\usepackage{caption,subcaption}

\usepackage{hyperref}
\usepackage[capitalise,nameinlink,sort]{cleveref}

\usepackage[T1]{fontenc}
\usepackage[utf8]{inputenc}  % Codage UNICODE (UTF-8)

\usepackage{mleftright}       % fix to wrong spacing of \left-,
\mleftright                   % \middle- \right-commands 

\declaretheorem[numberwithin=section]{theorem}
\declaretheorem[sibling=theorem]{lemma,corollary,proposition}
\declaretheorem[sibling=theorem,style=definition]{example,definition,remark,observation,sketch}

\declaretheorem[name=Question, refname={Question,Questions}, style=definition]{question}
\declaretheorem[name=Property, refname={Property,Properties}, style=definition]{property}
\declaretheorem[name=Corollary, refname={Corollary,Corollaries}, style=definition, numbered=no]{corollary-nonum}
\declaretheorem[name=Problem, refname={Problem,Problems}, style=definition]{problem}
\crefname{equation}{}{}

\newcommand{\ttt}[1]{\texttt{#1}}
\newcommand{\mtt}[1]{\mathtt{#1}}

\DeclareMathOperator{\Per}{Per}
\DeclareMathOperator{\N}{\mathbb{N}}
\DeclareMathOperator{\Z}{\mathbb{Z}}
\DeclareMathOperator{\D}{\mathbb{D}}
\DeclareMathOperator{\rep}{rep}
\DeclareMathOperator{\val}{val}
\newcommand{\len}[1]{\left|#1\right|}
\newcommand{\floor}[1]{\left\lfloor#1\right\rfloor}

\begin{document}

\title{Positionality of Dumont--Thomas numeration systems for integers}

\author{Savinien  Kreczman }

\author{Manon Stipulanti}

\affil{Department of Mathematics, University of Li\`{e}ge\\
All\'{e}e de la D\'{e}couverte 12, 4000 Li\`{e}ge, Belgium\\
\url{{savinien.kreczman,m.stipulanti}@uliege.be}
}

\author{S\'{e}bastien Labb\'{e}}

\affil{Universit\'{e} de Bordeaux,  CNRS, Bordeaux INP, LaBRI,  UMR 5800,\\
  F-33400, Talence, France \\
\url{sebastien.labbe@labri.fr}}

\maketitle            
\begin{abstract}
Introduced in 2001 by Lecomte and Rigo,  abstract numeration systems provide a way of expressing natural numbers with words from a language $L$ accepted by a finite automaton.
As it turns out,  these numeration systems are not necessarily positional,  i.e.,  we cannot always find a sequence $U=(U_i)_{i\ge 0}$ of integers such that the value of every word in the language $L$ is determined by the position of its letters and the first few values of $U$.
Finding the conditions under which an abstract numeration system is positional seems difficult in general.
In this paper,  we thus consider this question for a particular sub-family of abstract numeration systems called Dumont--Thomas numeration systems. 
They are derived from substitutions and were introduced in 1989 by Dumont and Thomas.
We exhibit conditions on the underlying substitution so that the corresponding Dumont--Thomas numeration is positional.
We first work in the most general setting,  then particularize our results to some practical cases.
Finally,  we link our numeration systems to existing literature,  notably properties studied by R\'{e}nyi in 1957,  Parry in 1960, Bertrand-Mathis in 1989,  and Fabre in 1995.

\keywords{morphism,  substitution,  periodic point,  numeration system,  positionality, Bertrand property,  Fabre substitution.}
\end{abstract}

\section{Introduction}\label{sec: intro}

The need of representing natural numbers (or expressing them in writing) has occupied humans for centuries. 
In its most general form,  this leads to so-called \emph{abstract numeration systems},  introduced by Lecomte and Rigo in 2001~\cite{Lecomte-Rigo-2001} (see also~\cite[Chapter~3]{CANT10} for a general presentation).
Such a numeration system is defined by a triple $S=(L,A,\prec)$ where $A$ is an alphabet ordered by the total order $\prec$ and $L$ is an infinite \emph{regular} language over $A$ (i.e.,  accepted by a deterministic finite automaton). 
We say that $L$ is the \emph{numeration language} of $S$.
When we order the words of $L$ with the \emph{genealogical} order (i.e., first by length, then using the dictionary order) induced by $\prec$,  we obtain a one-to-one correspondence $\rep_S$ between $\N$ and $L$.
The \emph{($S$-)representation} of the non-negative integer $n$ is then the $(n+1)$st word of $L$, and the inverse map, called the \emph{($S$-)evaluation map}, is denoted by $\val_S$.
A simple example is given by the abstract numeration system $S$ built on the language $L=\ttt{1}^* \ttt{2}^*$ over the ordered alphabet $\{\ttt{1},\ttt{2}\}$.
The first few words in the language are $\varepsilon,\ttt{1},\ttt{2},\ttt{11},\ttt{12},\ttt{22},\ttt{111}$. 
For instance, $\rep_S(5)=\ttt{22}$ and $\val_S(\ttt{111})=6$.

In general,  a numeration system $S=(L,A,\prec)$ is \emph{positional} if the underlying alphabet $A$ is a set of consecutive integers $\{\ttt{0},\ttt{1},\ldots, \ttt{c}\}$ for some $\ttt{c}\in\N$ and if there exists a sequence $(U_i)_{i\ge 0}$ of non-negative integers such that the evaluation map is of the form $\val_S \colon A^* \to \N, w_{k-1}\cdots w_0 \mapsto \sum_{i=0}^{k-1} w_i U_i$.
We use the term \emph{positional} because the \emph{positions} of letters in representations are used to generate numbers. 
Observe that the numeration system built on $L=\ttt{1}^* \ttt{2}^*$ cannot be positional: indeed,  we have $\rep_S(3)=\ttt{11}$ and $\rep_S(5)=\ttt{22}$,  so there is no hope to find an integer sequence $(U_i)_{i\ge 0}$ such that $3 = 1\cdot U_1 + 1\cdot  U_0$ and $5 = 2 \cdot U_1 + 2 \cdot U_0$ (see also~\cite[Example~3.1.12]{CANT10}).
We thus raise the following question (see also~\cite[Exercise~3.13]{CANT10}):

\begin{question}
\label{quest:ANS case}
What are the conditions for an abstract numeration system to be positional?
\end{question}

As this question seems difficult to answer in its full generality,  we consider a particular case with Question~\ref{quest:DT case} below.
Roughly,  we consider numeration systems that are derived from \emph{substitutions},  i.e.,  maps sending sequences to sequences and satisfying some mild properties.
These numeration systems are due to Dumont and Thomas~\cite{Dumont-Thomas-1989} in 1989.
(Note that to properly state the weaker version of Question~\ref{quest:ANS case},  we introduce some notation and definitions in~\cref{sec:prelim}.)
Since then,  Dumont--Thomas numeration systems have been used to solve various problems as they provide a nice framework to work with. 
As the literature is quite vast,  we focus on some recent papers from the 2020's only. 
For instance,  as Question~\ref{quest:DT case} already testifies,  Dumont--Thomas numerations fail to satisfy natural properties that are however somewhat desirable for numeration systems.
Furthermore,  they are typically neither greedy nor \emph{addable},  i.e.,  addition is not computable by a finite automaton. 
The latter condition is needed in particular if one wants to make use of the theorem-prover \verb!Walnut!~\cite{Walnut1,Walnut2}.
Recently,  a characterization of such addable numeration systems was found~\cite{Carton-Couvreur-Delacourt-Ollinger-2024}. 
Generalizations of Dumont--Thomas numeration systems as numeration systems \emph{per se} can be found in~\cite{MR4836876,Surer-2018,Surer-2020}.
In~\cite{Gheeraert-Romana-Stipulanti-2023,Gheeraert-Romana-Stipulanti-2024},  classical Dumont--Thomas numeration systems  are used to find string attractors for infinite words.
In~\cite{Miro-Rust-Sadun-Tadeo-2023}, the authors use extensions of Dumont--Thomas numeration systems to the setting of random substitutions.
In~\cite{Marshall-Maldonado-2024},  the specific case of the Thue--Morse substitution  is used to establish uniform bounds for the twisted correlations for all elements in the Thue--Morse subshift.

Thus Dumont--Thomas numeration systems form an interesting particular case of abstract numeration systems, where the extra structure brought by the existence of an underlying substitution makes the problem of positionality more tractable. Rather than Question~\ref{quest:ANS case} above, we study its restriction to this context, which was already asked in \cite{Lepsova-2024}.

\begin{restatable*}{question}{questionB}\label{quest:DT case}
{\rm\cite[Question 6.5.7]{Lepsova-2024}}
What are the conditions for a complement Dumont--Thomas numeration system to be positional?
\end{restatable*}

Here, as we will explain later, the term ``complement'' refers to an extension of Dumont--Thomas systems introduced in \cite{MR4836876}, where negative numbers also admit a representation.

We obtain a characterization of the substitutions $\mu$ that lead to a
positional system (\cref{thm: main}). The statement of this result is
technical, but takes a simpler form in two particular but interesting cases.
Those are stated as corollaries below. 

In the case of the original, ungeneralized Dumont--Thomas numeration
systems, we obtain the following corollary (the set $E_\mu$ is defined in
\cref{def: internal letter}; it is the set of all letters that appear in an
image by $\mu$ at some position other than the last).

\begin{restatable*}{corollary}{casoriginal}\label{cor: fixed point case}
Let $\mu\colon A^*\to A^*$ be a substitution and $u=u_0u_1\cdots$ be a right-infinite fixed point of $\mu$ with growing seed $a$. 
Then the Dumont--Thomas numeration system (for $\N$) associated with $\mu$, $u$ and $r=0$ is positional if and only if the map $c\mapsto \len{\mu^{\ell}(c)}$ is constant over $E_\mu$ for every $\ell$. 
In this case, the sequence of weights is given by $U_\ell=\len{\mu^\ell(a)}$ for every $\ell\ge 0$.
\end{restatable*}

In the case where the substitution is primitive, \cref{thm: main} can be
simplified to obtain the following corollary.

\begin{restatable*}{corollary}{casprimitif}\label{cor: primitive case}
Let $\mu \colon A^* \to A^*$ be a primitive substitution and let $u\in\Per_{\Z}(\mu)$ be a two-sided periodic point with growing seed $b|a$ and period $p\ge 1$. 
The Dumont--Thomas complement numeration system associated with $\mu$,  $u$,  and $r$ is positional if and only if the map $c\mapsto \len{\mu^{\ell}(c)}$ is constant over $E_\mu$ for every $\ell\ge 0$. In this case,  for every $\ell\ge 0$,  $U_\ell$ is the constant value of $c\mapsto \len{\mu^{\ell}(c)}$ over $E_\mu$ and $V_\ell=\len{\mu^\ell(b)}$.
\end{restatable*}

Furthermore, we find that positional Dumont--Thomas numeration systems are linked to other families of numeration systems existing in literature, namely positional greedy numeration systems and Bertrand numeration systems~\cite{Bertrand-Mathis-1989,Charlier-Cisternino-Stipulanti-2022}. Our main result is stated below; the necessary concepts are introduced in~\cref{sec:particularizations}.

\begin{restatable*}{corollary}{resumeliens}\label{cor: links}
Consider the case of the original Dumont--Thomas systems, that is, for right-infinite fixed points.

\begin{enumerate}[label=(\alph*)]
\item If the Dumont--Thomas system associated with $\mu$ is positional, it is equal to a Dumont--Thomas numeration system associated with a Fabre-like substitution $\nu$.
\item The numeration systems that are simultaneously equal to a Dumont--Thomas numeration system and to a $U$-system, are exactly the Bertrand numeration systems associated with Parry numbers. 

In particular, the sequence $U$ is a linear recurrence sequence in this case.
Furthermore, the substitution within the Dumont--Thomas system is primitive if and only if the Bertrand numeration system is canonical.
\end{enumerate}
\end{restatable*}

The structure of the paper is as follows.
In~\cref{sec:prelim},  we present the necessary background and preliminary results.
In particular,  we present the original generalization of Dumont--Thomas numerations to all integers in~\cref{sec:DT for N and Z}, then generalize it further in~\cref{sec:DT for N and Z with various residues} by letting the lengths of representations vary. \cref{def: DTNS for Z equiv r mod p} is the outcome of this generalization. A special feature of both this article and its predecessor \cite{MR4836876} is that the substitutions at play need not be primitive, which makes matters more complicated but more general. 

In~\cref{sec: positionality of DT}, we study which Dumont--Thomas numeration systems are positional to answer Question~\ref{quest:DT case}.
The main result is \cref{thm: main} in \cref{sec:mainresult}. This section introduces the technical notions required, states and proves the theorem. A sketch of the proof is provided in~\cref{sec: sketch}, that allows the reader to grasp the mechanism behind this proof without getting bogged down in technicalities. Then, \cref{ex: phi and phi-squared suite,ex:bca-bb-b} illustrate some of the difficulties that arise, notably due to the lack of primitivity.

\cref{sec:particularizations} is where we exploit our main result to deduce the corollaries stated above. We start with the simplified form our main result takes in the original and in the primitive case.
Then, we investigate the link between positional Dumont--Thomas and positional greedy or Bertrand numeration systems.

We conclude the article by mentioning directions for future research.

We note that the present paper is a long version of the article~\cite{Kreczman-Labbe-Stipulanti-2025} presented at the conference \emph{WORDS 2025}.
Compared to the conference version,  this paper contains all the proofs of the results, a more elaborate introduction to Dumont--Thomas numeration systems in \cref{sec:DT for N and Z,sec:DT for N and Z with various residues}, and more complete discussions throughout the paper and particularly in~\cref{sec:particularizations}.

\section{Preliminaries}\label{sec:prelim}

\textbf{General combinatorics on words}.
An \emph{alphabet} $A$ is a finite set and its elements are called \emph{letters}.
A (\emph{finite} or \emph{infinite}) \emph{word} over $A$ is a (finite or infinite) sequence of letters belonging to $A$.
For a finite word $w$ over $A$,  we let $|w|$ denote its \emph{length},  i.e.,  the number of letters it is made of. The \emph{empty word} is denoted by $\varepsilon$. 
For a word $w$ over $A$,  a \emph{factor} of $w$ is a word $y$ such that there exist words $x,z$ such that $w=xyz$.
A \emph{prefix} (resp. \emph{suffix}) is a factor $y$ such that $x$ (resp. $z$) is empty.
If $w=ps$ for some words $p,s$,  then we write $p^{-1}w=s$ and $ws^{-1}=p$. 
For a word $w$  we let $w_i$ denote its $i$-th letter.
We use the notation of intervals to indicate portions of words: if $I$ is an interval of integers,  we let $w_I$ denote the factor $(w_i)_{i\in I}$ of $w$. (Recall the difference between parentheses and square brackets to denote intervals.)
We let $A^*$ denote the set of all finite words over the alphabet $A$ and $A^+$ the set of all non-empty words over $A$ (note that $A^+ = A^* \setminus \{\varepsilon\}$).
For $\D\in\{\N,\Z, \Z_{<0}\}$,  we let  $A^{\D}$ the set of words indexed over $\D$.
We call $u_0 u_1 u_2 \cdots \in A^{\N}$ a \emph{right-infinite} word and $\cdots u_{-3} u_{-2} u_{-1} \in A^{\Z_{< 0}}$ a \emph{left-infinite} word.
A \emph{two-sided} word is a word indexed over $\Z$.
For convenience,  we separate by a vertical bar its $-1$-th and $0$-th elements to indicate the origin,  i.e.,   $u = \cdots u_{-3} u_{-2} u_{-1}|u_0u_1u_2 \cdots$.
If the alphabet $A$ is totally ordered by $\prec$,  we define the \emph{lexicographic order} on  $A^{*}$ and $A^{\N}$ as follows.
For two words $u,v\in A^*$,  we write $u \prec_{\text{lex}} v$ if $u$ is a proper prefix of $v$ or if there exist $p,s,s'\in A^*$ and $a,a'\in A$ such that $u=pas$, $v=pa's'$,  and $a \prec a'$.
For two words $u,v\in A^{\N}$,  we write $u \prec_{\text{lex}} v$ if there exist $p\in A^*$,  $x,x'\in A^{\N}$,  and $a,a'\in A$ such that  $u=pax$, $v=pa'x'$,  and $a \prec a'$.
For finite or infinite words $u,v$,  we write $u \preccurlyeq_{\text{lex}} v$ if $u \prec_{\text{lex}} v$ or $u=v$.

\textbf{Morphisms and substitutions}.
Given alphabets $A,B$,  a \emph{morphism} is a map $\mu \colon A^* \to B^*$ such that $\mu (uv) = \mu(u) \mu(v)$ for all words $u,v\in A^*$.
A morphism is entirely determined by the images of the letters of $A$.
A \emph{substitution} is a morphism $\mu \colon A^* \to A^*$ such that the image $\mu(a)$ is non-empty for every letter $a\in A$ and there exists a \emph{growing} letter $a\in A$,  i.e.,  $\lim_{n\to + \infty} |\mu^n(a)| = + \infty$.
A morphism $\mu \colon A^* \to A^*$ is \emph{primitive} if there exists an integer $k\in\N$ such that for all $a, b\in A$,  the letter $a$ appears in $\mu^k(b)$.
Note that,  if we consider the classical adjacency matrix $M$ of $\mu \colon A^* \to A^*$,  $\mu$ is  primitive if and only if $M$ is primitive. From the theory of primitive matrices, we then deduce that if a morphism is primitive, then there exists an integer $k\in\N$ such that for all $a, b\in A$ and for all $\ell\geq k$,  the letter $a$ appears in $\mu^\ell(b)$.

With a morphism $\mu \colon A^* \to A^*,$  we can associate a directed graph in the following way: the nodes are the letters of $A$ and for every $a\in A$,  if we write $\mu(a)=c_0 \cdots c_\ell$ with $c_i \in A$ for every $i\in\{0,\ldots,\ell\}$,  then we draw an arrow labeled by $i$ from $a$ to $c_i$ for every $i\in\{0,\ldots,\ell\}$.
This directed graph encodes the images of letters under $\mu$.
Similarly,  we may unfold the directed graph into a directed tree starting with some fixed vertex that is then called the \emph{root} of the tree. 
Given $a\in A$,  we define the tree $\mathcal{T}_{\mu,a}$ as follows: its root is labeled by $a$, and if a node of the tree is labeled by $x$ and $\mu(x)=y_0\cdots y_\ell$ then that node has $\ell+1$ children labeled $y_0$ to $y_\ell$ from left to right, with the edge from $x$ to $y_i$ being labeled by $i$ with $0\le i \le \ell$. We say that $x$ is the \emph{parent} of the node $y_i$, that $y_0,\ldots,y_{i-1},y_{i+1},\ldots,y_{\ell}$ are its \emph{siblings}, among which $y_{i+1},\ldots,y_{\ell}$ are its \emph{younger siblings}. 
Observe that the $k$-th level of $\mathcal{T}_{\mu,a}$ stores the $k$-th iteration of $\mu$ on the letter $a$.
Given an integer $n\ge 0$,  we say that a node of the tree is \emph{in column n} if there are $n$ nodes of the same level to its left. See~\cref{fig:tree-introduction} for an illustration of those terms.

\begin{figure}
\centering
\begin{tikzpicture}[xscale=.75, yscale=1, >=latex,
node0/.style={},
nodeM/.style={fill=cyan!50},
nodeA/.style={fill=orange!50}]
%positive nodes
    \node[rectangle,draw,node0] (var) at (0,0)  {$a$};
    \node[rectangle,draw,node0] (0) at (0,-1) {$a$};
    \node[rectangle,draw,node0] (1) at (1,-1) {$b$};
    \node[rectangle,draw,node0] (00) at (0,-2) {$a$};
    \node[rectangle,draw,node0] (01) at (1,-2) {$b$};
    \node[rectangle,draw,nodeM] (10) at (2,-2) {$a$};
    \node[rectangle,draw,node0] (000) at (0,-3) {$a$};
    \node[rectangle,draw,node0] (001) at (1,-3) {$b$};
    \node[rectangle,draw,node0] (010) at (2,-3) {$a$};
    \node[rectangle,draw,nodeA] (100) at (3,-3) {$a$};
    \node[rectangle,draw,node0] (101) at (4,-3) {$b$};
    \node[] (zer) at (0,-3.5)  {$0$};
    \node[] (un) at (1,-3.5)  {$1$};
    \node[] (2) at (2,-3.5)  {$2$};
    \node[] (3) at (3,-3.5)  {$3$};
    \node[] (4) at (4,-3.5)  {$4$};
    %positive edges
    \draw[->] (var)  -- node[fill=white,inner sep=2pt] {\scriptsize \ttt{0}} (0);
    \draw[->] (var)  -- node[fill=white,inner sep=2pt] {\scriptsize \ttt{1}} (1);
    \draw[->] (0) -- node[fill=white,inner sep=2pt] {\scriptsize \ttt{0}} (00);
    \draw[->] (0) -- node[fill=white,inner sep=2pt] {\scriptsize \ttt{1}} (01);
    \draw[->] (1) -- node[fill=white,inner sep=2pt] {\scriptsize \ttt{0}} (10);
     \draw[->] (00) -- node[fill=white,inner sep=2pt] {\scriptsize \ttt{0}} (000);
     \draw[->] (00) -- node[fill=white,inner sep=2pt] {\scriptsize \ttt{1}} (001);
     \draw[->] (01) -- node[fill=white,inner sep=2pt] {\scriptsize \ttt{0}} (010);
     \draw[->] (10) -- node[fill=white,inner sep=2pt] {\scriptsize \ttt{0}} (100);
      \draw[->] (10) -- node[fill=white,inner sep=2pt] {\scriptsize \ttt{1}} (101);
\end{tikzpicture}
\caption{The first few levels of the tree associated with the Fibonacci substitution $\varphi\colon a\mapsto ab,\, b\mapsto a$.  The orange node has a younger sibling labeled $b$ and its parent is the blue node. Below, numbers indicate in which column each node stands. There are five nodes at level $3$, occupying the columns $0$ to $4$.}
\label{fig:tree-introduction}
\end{figure}

Substitutions can naturally be applied to two-sided words by setting
\[
\mu(\cdots u_{-3} u_{-2} u_{-1}|u_0u_1u_2 \cdots)
=
\cdots \mu(u_{-3})  \mu(u_{-2}) \mu(u_{-1}) |\mu(u_0) \mu(u_1) \mu(u_2) \cdots
\]
Let $\D\in\{\N,\Z, \Z_{<0}\}$ and consider a substitution $\mu$ over $A$.
A word $u\in A^{\D}$ is a \emph{periodic point of $\mu$} if there exists an integer $p\ge 1$ such that $\mu^p(u)=u$. 
In this case,  $p$ is called a \emph{period} of the periodic point $u$.
The smallest such integer is called the \emph{period of $u$}. 
A periodic point of $\mu$ with period $p = 1$ is called a \emph{fixed point} of $\mu$. 
We let $\Per_{\D}(\mu)= \{u\in A^{\D} \mid \mu^p(u)=u \text{ for some $p\ge 1$} \}$ denote the set of periodic points of $\mu$.
If $u \in \Per_{\Z}(\mu)$,  then the \emph{seed} of $u$ is the pair of letters $u_{-1}|u_0$; see~\cite[\S 4.1]{Baake-Grimm-2013}. 
If both letters of the seed of a two-sided periodic
point are growing, then the periodic point is defined entirely by its seed. 
More precisely,  we have $u = \lim_{n\to+\infty} \mu^{np}(u_{-1})|\lim_{n\to+\infty} \mu^{np}(u_0)$, where $p$ is a period of $u$.
By extension of the tree associated with a substitution and a letter as in the previous paragraph,  we consider two-sided trees as follow. 
For a substitution $\mu$ over $A$ and two letters $a,b\in A$,  the tree $\mathcal{T}_{\mu,b|a}$ is obtained by setting a start root having two children: the left one is reached with an edge of label $\ttt{1}$ and is the root of the tree $\mathcal{T}_{\mu,b}$,  and the right one is reached with an edge of label $\ttt{0}$ and is the root of the tree $\mathcal{T}_{\mu,a}$ (see~\cref{fig: double tree} for an illustration).  For an integer $n$,  we say that a node is \emph{in column $n$} if either $n\geq 0$, the node is the right subtree and there are $n$ nodes on the same level to its left in the right subtree, or $n<0$, the node is in the left subtree and there are $-n-1$ nodes on the same level to its right in the left subtree.
Examples of such trees are given later in the paper. See for instance~\cref{fig:DTNS-mu-u-tree}.

\begin{figure}
\centering
\begin{tikzpicture}[xscale=1, yscale=1.2, >=latex,
node0/.style={},
nodeM/.style={fill=cyan!50},
edgeLabel/.style={fill=white, inner sep=1pt}]

% Root
    \node                 (O)   at (0,0.5) {};
    \node[rectangle,draw,nodeM] (root) at (0,0) {start};
    \draw[->] (O) to (root);

% Children (no visible nodes, just coordinates)
\node (Tb) at (-1.6,-1.8) {};
\node (Ta) at (1.7,-1.8) {};

% Edges with labels (slightly closer to root)
\draw[->] (root) -- node[edgeLabel, pos=0.2] {\scriptsize \ttt{1}} (Tb);
\draw[->] (root) -- node[edgeLabel, pos=0.2] {\scriptsize \ttt{0}} (Ta);

% Labels closer to blobs
\node at (-2.1,-1.8) {$\mathcal{T}_{\mu,b}$};
\node at (2.1,-1.8) {$\mathcal{T}_{\mu,a}$};

% Background blobs (smaller radius)
\begin{scope}[on background layer]

% Blob around Tb (reduced scale)
\draw[dashed]
  ($(Tb)+(-0.9,0.6)$)
  .. controls ($(Tb)+(-1.2,0)$) and ($(Tb)+(-1,-0.8)$)
  .. ($(Tb)+(0,-0.9)$)
  .. controls ($(Tb)+(0.8,-0.7)$) and ($(Tb)+(1,0)$)
  .. ($(Tb)+(0.6,0.8)$)
  .. controls ($(Tb)+(0,1)$) and ($(Tb)+(-0.6,0.7)$)
  .. cycle;

% Blob around Ta (reduced scale)
\draw[dashed]
  ($(Ta)+(-0.9,0.6)$)
  .. controls ($(Ta)+(-1.2,0)$) and ($(Ta)+(-1,-0.8)$)
  .. ($(Ta)+(0,-0.9)$)
  .. controls ($(Ta)+(0.8,-0.7)$) and ($(Ta)+(1,0)$)
  .. ($(Ta)+(0.6,0.8)$)
  .. controls ($(Ta)+(0,1)$) and ($(Ta)+(-0.6,0.7)$)
  .. cycle;

\end{scope}
\end{tikzpicture}
\caption{The tree $\mathcal{T}_{\mu,b|a}$ associated with a substitution $\mu$ over $A$ and two letters $a,b\in A$.}
\label{fig: double tree}
\end{figure}

\textbf{Numeration systems}.
A \emph{numeration system} over the \emph{domain} $\D\in\{\N,\Z_{<0},\Z\}$ is a pair of maps between $\D$ and a set of words,  i.e.,  the \emph{representation map} $\rep\colon \D\to A^*$ for some alphabet $A$, and the \emph{evaluation map} $\val\colon L\to \D$, where $\rep(\D)\subset L\subset A^*$, such that $\val\circ\rep=id_{\D}$. The set $\rep(\D)$ is the \emph{language} of the numeration system, while $L$ contains some additional, non-canonical representations. 
In general,  a numeration system over $\D=\N$ is \emph{positional} if the underlying alphabet $A$ is a set of consecutive integers $\{\ttt{0},\ttt{1},\ldots, \ttt{c}\}$ for some $\ttt{c}\in\N$ and the evaluation map is of the form $\val \colon A^* \to \N, w_{k-1}\cdots w_0 \mapsto \sum_{i=0}^{k-1} w_i U_i$ for some sequence $U=(U_i)_{i\ge 0}\in\N^{\N}$. 
Over $\D=\Z$,  a numeration system is \emph{positional} if the underlying alphabet $A$ is a set of consecutive integers $\{\ttt{0},\ttt{1},\ldots, \ttt{c}\}$ for some $\ttt{c}\in\N$ and the evaluation map is of the form $\val \colon A^* \to \Z, w_{k-1}\cdots w_0 \mapsto \sum_{i=0}^{k-2} w_i U_i - w_{k-1}V_{k-1}$ for some sequences $U=(U_i)_{i\ge 0},V=(V_i)_{i\ge 0}\in\N^{\N}$. 
The sequences $U$ and $V$ are the sequences of \emph{weights} of the numeration system. 
Every position has a given weight, while the presence of an additional sequence $V$ helps deal with the representation of negative numbers, in a fashion similar to the usual two's complement numeration system described in the next example.

\begin{example}
The two's complement numeration system allows representations of all integers in a binary system using powers of $2$.
Let $A=\{\ttt{0},\ttt{1}\}$.
The evaluation map is defined as follows: for a word $w=w_{k-1}w_{k-2} \cdots w_0$ over $A$,  we set $\val_{2c}(w)=-w_{k-1} 2^{k-1} + \sum_{i=0}^{k-2} w_i 2^i$. 
Observe now that,  for any word $w$ over $A$,  we have $\val_{2c}(\ttt{00}w)=\val_{2c}(\ttt{0}w)$ and $\val_{2c}(\ttt{11}w)=\val_{2c}(\ttt{1}w)$.
So for any integer $n\in\Z$,  there exists a unique word $w\in A^*\setminus(\ttt{00}A^*\cup\ttt{11}A^*)$ such that $n=\val_{2c}(w)$.
We let $\rep_{2c}(n)$ denote this unique word and we call it the \emph{two's complement representation} of $n$.
The first few two's complement representations of integers are $(\rep_{2c}(n))_{-4\le n\le 4}=\ttt{100},  \ttt{101}, \ttt{10},  \ttt{1},  \varepsilon,  \ttt{01},  \ttt{010}, \ttt{011},  \ttt{0100}$.
Observe that the two's complement numeration system is positional: we may use the sequence of powers of $2$ as weights.
\end{example}

\subsection{Dumont--Thomas numeration systems}
\label{sec:DT for N and Z}

In this section,  we expose the theory of Dumont and Thomas on numeration systems~\cite{Dumont-Thomas-1989}.
Roughly,  a substitution $\mu$ defines a numeration system where the representation of an integer $n$ is the label of a path from the root to column $n$ in a tree associated with $\mu$.

We first recall the notion of admissible sequences for a substitution.
Let $\mu \colon A^*\to A^*$ be a substitution. 
Fix a letter $a\in A$ and an integer $k$.
For every $i \in \{0,\ldots,k\}$,  we consider a pair $(m_i,a_i) \in A^*\times A$.
The sequence $((m_i,a_i))_{i=0,\dots,k}$ is \emph{admissible with respect to $\mu$} if for every $i \in \{1,\ldots,k\}$,  $m_{i-1}a_{i-1}$ is a prefix of $\mu(a_i)$. 
This sequence is \emph{$a$-admissible with respect to $\mu$} if it is admissible with respect to $\mu$ and $m_ka_k$ is a prefix of $\mu(a)$.
When the context is clear,  we simply say \emph{admissible} or \emph{$a$-admissible} without
specifying the substitution.
See~\cref{fig: a-admissible illustration} for an illustration.

\begin{figure}
\centering
\begin{tikzpicture}
%level 0
\node[] (a) at (-1,0.2)  {$\mu^0(a)$};
\draw [draw=black] (0,0) rectangle (0.4,0.4);
\node[] (a) at (0.2,0.2)  {$a$};
\draw[dashed]   (0,0)  -- (0,-1);
\draw[dashed]   (0.4,0)  -- (2.7,-0.6);
%level 1
\node[] (a) at (-1,-0.8)  {$\mu(a)$};
\draw [draw=black] (0,-1) rectangle (2.7,-0.6);
\node[] (a) at (0.5,-0.8)  {$m_2$};
\draw [draw=black] (0,-1) rectangle (1,-0.6);
\draw[dashed]   (0,-1)  -- (0,-1.6);
\node[] (a) at (1.2,-0.8)  {$a_2$};
\draw [draw=black] (1,-1) rectangle (1.4,-0.6);
\draw[dashed]   (1,-1)  -- (1.5,-1.6);
\draw[dashed]   (1.4,-1)  -- (3,-1.6);
\draw[dashed]   (2.7,-1)  -- (4.1,-1.6); %end
%level 2
\node[] (a) at (-1,-1.8)  {$\mu^2(a)$};
\draw [draw=black] (0,-2) rectangle (4.1,-1.6);
\node[] (a) at (0.75,-1.8)  {$\mu(m_2)$};
\node[] (a) at (1.8,-1.8)  {$m_1$};
\draw [draw=black] (1.5,-2) rectangle (2.1,-1.6);
\draw[dashed]   (0,-2)  -- (0,-2.6);
\draw[dashed]   (1.5,-2)  -- (2,-2.6);
\node[] (a) at (2.3,-1.8)  {$a_1$};
\draw [draw=black] (2.1,-2) rectangle (2.5,-1.6);
\draw[dashed]   (2.1,-2)  -- (3,-2.6);
\draw[dashed]   (2.5,-2)  -- (4.8,-2.6);
\draw[dashed]   (4.1,-2)  -- (5.8,-2.6); %end
%level 3
\node[] (a) at (-1,-2.8)  {$\mu^3(a)$};
\draw [draw=black] (0,-3) rectangle (5.8,-2.6);
\node[] (a) at (1,-2.8)  {$\mu^2(m_2)$};
\draw [draw=black] (2,-3) rectangle (3,-2.6);
\node[] (a) at (2.5,-2.8)  {$\mu(m_1)$};
\node[] (a) at (3.5,-2.8)  {$m_0$};
\draw [draw=black] (3,-3) rectangle (4,-2.6);
\node[] (a) at (4.2,-2.8)  {$a_0$};
\draw [draw=black] (4,-3) rectangle (4.4,-2.6);
\end{tikzpicture}
\caption{An illustration of what it means for a sequence $((m_i,a_i))_{0\le i \le 2}$ to be $a$-admissible with respect to a substitution $\mu$.}
\label{fig: a-admissible illustration}
\end{figure}
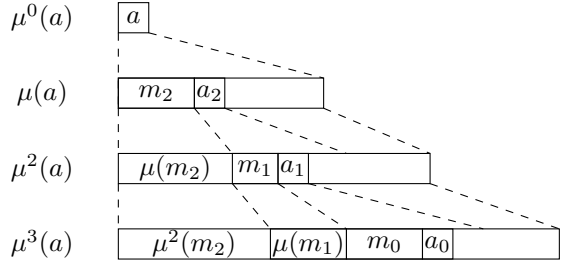

\begin{example}
\label{ex:Trib}
Consider the Tribonacci substitution $\tau \colon a \mapsto ab,  b \mapsto ac,  c \mapsto a$.
The tree $\mathcal{T}_{\tau,a}$ in~\cref{fig:DTNS-for-Trib} is defined as in~\cref{sec:prelim} and allows us to compute $a$-admissible sequences.
For instance,  the sequence $(m_0,a_0)=(\varepsilon,a),(m_1,a_1)=(\varepsilon,a),(m_2,a_2)=(a,b)$ is $a$-admissible and corresponds to the path in bold in $\mathcal{T}_{\tau,a}$.
Note that the sequences $(a,c),(a,b)$ and $(a,c),(a,b),(\varepsilon,a)$ are also $a$-admissible and lead to nodes labeled by $c$ and colored (in orange) in $\mathcal{T}_{\tau,a}$,  both in column $3$.                                                                                                                                                                                                                                                                                                                                                                                                                                                                                                                                                                                                                 
\begin{figure}[h]
\centering
\subfloat[\label{fig:DTNS-for-Trib-orientedgraph}]
{
\begin{tikzpicture}[scale=1.6, >=latex,
node0/.style={},
nodeM/.style={fill=cyan!50},
nodeA/.style={fill=orange!50}]
    % nodes positive
    \node[rectangle,draw] (a) at (.5,1) {$a$};
    \node[rectangle,draw] (b) at (0,0) {$b$};
    \node[rectangle,draw] (c) at (1,0) {$c$};
    % positive edges
    \draw[->,loop left] (a) to node[left] {\scriptsize $0$} (a);
 \draw[->,bend right=15] (a) to node[fill=white,inner sep=2pt] {\scriptsize $1$} (b);
    \draw[->,bend right=15] (b) to node[fill=white,inner sep=2pt] {\scriptsize $0$} (a);
        \draw[->] (b) to node[fill=white,inner sep=2pt] {\scriptsize $1$} (c);
    \draw[->] (c) to node[fill=white,inner sep=2pt] {\scriptsize $0$} (a);
\end{tikzpicture}
%%%
}
\hspace{1cm}
\subfloat[\label{fig:DTNS-for-Trib-tree}]
{
%%%
\begin{tikzpicture}[xscale=.75, yscale=1.4, >=latex,
node0/.style={},
nodeM/.style={fill=cyan!50},
nodeA/.style={fill=orange!50}]
%positive nodes
    \node[rectangle,draw,node0] (var) at (0,0)  {$a$};
    \node[rectangle,draw,node0] (0) at (0,-1) {$a$};
    \node[rectangle,draw,node0] (1) at (1,-1) {$b$};
    \node[rectangle,draw,node0] (00) at (0,-2) {$a$};
    \node[rectangle,draw,node0] (01) at (1,-2) {$b$};
    \node[rectangle,draw,node0] (10) at (2,-2) {$a$};
    \node[rectangle,draw,nodeA] (11) at (3,-2) {$c$};
    \node[rectangle,draw,node0] (000) at (0,-3) {$a$};
    \node[rectangle,draw,node0] (001) at (1,-3) {$b$};
    \node[rectangle,draw,node0] (010) at (2,-3) {$a$};
    \node[rectangle,draw,nodeA] (011) at (3,-3) {$c$};
    \node[rectangle,draw,node0] (100) at (4,-3) {$a$};
    \node[rectangle,draw,node0] (101) at (5,-3) {$b$};
    \node[rectangle,draw,node0] (110) at (6,-3) {$a$};
    \node[] (zer) at (0,-3.5)  {$0$};
    \node[] (un) at (1,-3.5)  {$1$};
    \node[] (2) at (2,-3.5)  {$2$};
    \node[] (3) at (3,-3.5)  {$3$};
    \node[] (4) at (4,-3.5)  {$4$};
    \node[] (5) at (5,-3.5)  {$5$};
    \node[] (6) at (6,-3.5)  {$6$};
    %positive edges
    \draw[->] (var)  -- node[fill=white,inner sep=2pt] {\scriptsize \ttt{0}} (0);
    \draw[->,line width=0.4mm] (var)  -- node[fill=white,inner sep=2pt] {\scriptsize \ttt{1}} (1);
    \draw[->] (0) -- node[fill=white,inner sep=2pt] {\scriptsize \ttt{0}} (00);
    \draw[->] (0) -- node[fill=white,inner sep=2pt] {\scriptsize \ttt{1}} (01);
    \draw[->,line width=0.4mm] (1) -- node[fill=white,inner sep=2pt] {\scriptsize \ttt{0}} (10);
    \draw[->] (1) -- node[fill=white,inner sep=2pt] {\scriptsize \ttt{1}}(11);
     \draw[->] (00) -- node[fill=white,inner sep=2pt] {\scriptsize \ttt{0}} (000);
     \draw[->] (00) -- node[fill=white,inner sep=2pt] {\scriptsize \ttt{1}} (001);
     \draw[->] (01) -- node[fill=white,inner sep=2pt] {\scriptsize \ttt{0}} (010);
     \draw[->] (01) -- node[fill=white,inner sep=2pt] {\scriptsize \ttt{1}} (011);
     \draw[->,line width=0.4mm] (10) -- node[fill=white,inner sep=2pt] {\scriptsize \ttt{0}} (100);
      \draw[->] (10) -- node[fill=white,inner sep=2pt] {\scriptsize \ttt{1}} (101);
       \draw[->] (11) -- node[fill=white,inner sep=2pt] {\scriptsize \ttt{0}} (110);
\end{tikzpicture}
%%%
}
\caption{On the left,  the directed graph associated with the Tribonacci substitution $\tau\colon a\mapsto ab,  b\mapsto ac, c\mapsto a$. 
On the right,  the tree $\mathcal{T}_{\tau,a}$ displays all $a$-admissible sequences of length at most $3$ for the Tribonacci substitution $\tau$ and growing letter $a$. 
}
\label{fig:DTNS-for-Trib}
\end{figure}
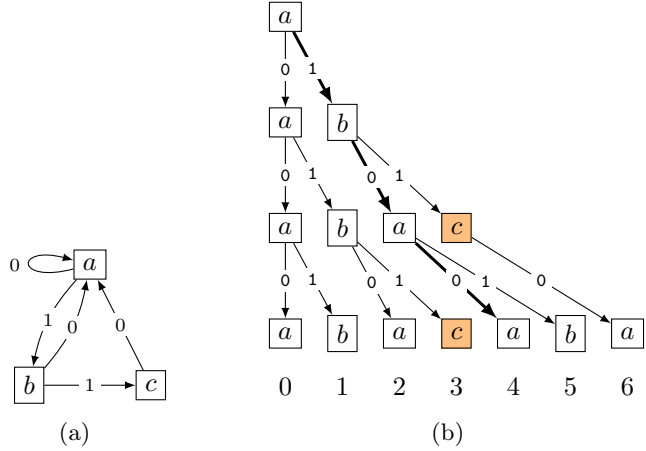
\end{example}

As illustrated in~\cref{ex:Trib},  given a right-infinite fixed point of a substitution, multiple paths in the corresponding tree can lead to nodes in the same column.
However,  there is a unique path starting from the root of the tree and labeled by a word not starting with $0$.
The following result,  shown by Dumont and Thomas,  essentially formalizes this and states that there is a unique admissible sequence for every prefix of the fixed point.

\begin{theorem}\label{thm:dt-unique-adminseq}
    {\rm\cite[Theorem 1.5]{Dumont-Thomas-1989}}
    Let $a\in A$ and let $\mu:A^*\to A^*$ be a substitution. 
    Let $u$ be a right-infinite fixed point of $\mu$ with growing seed $u_0 = a$.
    For every integer $n\geq1$,
    there exist a unique integer $k=k(n)$ and a unique sequence
    $((m_i,a_i))_{i=0,\dots,k-1}$ such that the sequence is $a$-admissible, $m_{k-1}\neq\varepsilon$,  and $ u_0u_1\cdots u_{n-1} = \mu^{k-1}(m_{k-1}) \mu^{k-2}(m_{k-2}) \cdots \mu^0(m_0)$.
    \end{theorem}

\cref{thm:dt-unique-adminseq} allows to define a numeration system for non-negative integers as defined below.

\begin{definition}
\label{def: DTNS for N}
Let $\mu \colon A^* \to A^*$ be a substitution,  $a$ be a growing letter,  and $u$ be a right-infinite fixed point of $\mu$ with seed $u_0=a$.
Set $\ttt{c}=\max_{b\in A}|\mu(b)|-1$ and define the set $D=\{\ttt{0},\ttt{1},\ldots,\ttt{c}\}$.
We define the map $\rep_{\mu,a} \colon \N \to D^*, n \mapsto \rep_{\mu,a} (n)$ by
\[
\rep_{\mu,a} (n) = 
\begin{cases}
\varepsilon,  & \text{if $n=0$}; \\
|m_{k-1}| \cdot |m_{k-2}| \cdots |m_0|,  & \text{if $n\ge 1$};
\end{cases}
\]
where $k=k(n)$ is the unique integer and $((m_i,a_i))_{i=0,\dots,k-1}$ is the unique sequence from~\cref{thm:dt-unique-adminseq}.
This numeration system is called the \emph{Dumont--Thomas numeration system associated with $\mu$ and $u$}.
\end{definition}

Note that we use a special font for the digits representing numbers to distinguish them from integers. Note also that $\rep_{\mu,a}(n)$ is the label of the shortest path from the root to a node in column $n$ in the tree $\mathcal{T}_{\mu,a}$.

\begin{example}
We resume~\cref{ex:Trib}.
The representation of the first few non-negative integers for the Tribonacci substitution $\tau$ and growing letter $a$ are given in~\cref{tab:DTNS-for-Trib-table}.
\begin{table}
\centering
\caption{The representation of the first few non-negative integers in the Dumont--Thomas numeration system for $\tau$ and $a$,  where $\tau\colon a\mapsto ab,  b\mapsto ac, c\mapsto a$ is the Tribonacci substitution.}
\label{tab:DTNS-for-Trib-table}
\begin{tabular}{c|ccccccc}
 		$n$ & $0$ & $1$ & $2$ & $3$ & $4$ & $5$ & $6$ \\
 		\hline 
 		$\rep_{\tau,a}(n)$& $\varepsilon$ & \ttt{1} & \ttt{10} & \ttt{11} & \ttt{100} & \ttt{101} & \ttt{110}
\end{tabular}
\end{table}
\end{example}

In~\cite{MR4836876},  the authors proposed an extension of the numeration system given in~\cref{def: DTNS for N} to all integers.
The next two results generalize~\cref{thm:dt-unique-adminseq} to right- and left-infinite periodic points of substitutions with a growing letter.

\begin{theorem}\label{thm: unique admin for right infinite periodic}
 {\rm\cite[Theorem 4.1]{MR4836876}}
Let $\mu:A^*\to A^*$ be a substitution with growing letter $a\in A$.
Consider a right-infinite periodic point $u\in \Per_{\N}(\mu)$ with $u_0=a$ and period $p\ge 1$.
For every integer $n\geq1$, there exist a unique integer $k=k(n)$ such that $p$ divides $k$ and a unique sequence $((m_i,a_i))_{i=0,\dots,k-1}$ such that the sequence is $a$-admissible,  $m_{k-1}m_{k-2} \cdots m_{k-p}\neq\varepsilon$,  and $ u_0u_1\cdots u_{n-1} = \mu^{k-1}(m_{k-1})\mu^{k-2}(m_{k-2}) \cdots \mu^0(m_0)$.
\end{theorem}
    
\begin{theorem}\label{thm: unique admin for left infinite periodic}
     {\rm\cite[Theorem 4.2]{MR4836876}}
Let $\mu:A^*\to A^*$ be a substitution with growing letter $b\in A$.
Consider a left-infinite periodic point $u\in \Per_{\Z_{<0}}(\mu)$ with $u_{-1}=b$ and period $p\ge 1$.
For every integer $n\leq -2$, there exist a unique integer $k=k(n)$ such that $p$ divides $k$ and a unique sequence $((m_i,a_i))_{i=0,\dots,k-1}$ such that
\begin{enumerate}
\item the sequence $((m_i,a_i))_{i=0,\dots,k-1}$ is $b$-admissible,
\item we have $-|\mu^k(b)|\leq n$, 
\item we have $\mu^{p-1}(m_{k-1})\mu^{p-2}(m_{k-2}) \cdots \mu^0(m_{k-p}) a_{k-p} \neq \mu^p(b)$,
\item and we have $u_{-|\mu^k(b)|} \cdots u_{n-2} u_{n-1} = \mu^{k-1}(m_{k-1})\mu^{k-2}(m_{k-2}) \cdots \mu^0(m_0)$.
\end{enumerate}
\end{theorem}

Note that the second condition $-|\mu^k(b)|\leq n$ was added from \cite[Theorem 4.2]{MR4836876}. This does not compromise the result as this condition is already imposed in the proof, but it avoids a degenerate case where we select $k$ such that $-|\mu^k(b)|> n$ and the admissible sequence with $m_{k-1}=\cdots=m_0=\varepsilon$, leading to the equality $\varepsilon=\varepsilon$ and to the loss of uniqueness.
    
Similarly to~\cref{def: DTNS for N},  the combination of~\cref{thm: unique admin for right infinite periodic,thm: unique admin for left infinite periodic} allows to define a numeration system for all integers as defined below.

\begin{definition}
\label{def: DTNS for Z}
 {\rm\cite[Definition 4.3]{MR4836876}}
Let $\mu \colon A^* \to A^*$ be a substitution and let $u\in\Per_{\Z}(\mu)$ be a two-sided periodic point with growing seed $u_{-1}|u_0$ and period $p\ge 1$. 
Define $\ttt{c}=\max_{a\in A}|\mu(a)|-1$  and the set $D=\{\ttt{0},\ttt{1},\ldots,\ttt{c}\}$.
We define the map $\rep_{u} \colon \Z \to \{\ttt{0},\ttt{1}\}D^*, n \mapsto \rep_{u} (n)$ by
\[
\rep_{u} (n) = 
\begin{cases}
\ttt{0} \cdot |m_{k-1}| \cdot |m_{k-2}| \cdots |m_0|,  & \text{if $n\ge 1$};\\
\ttt{0},  & \text{if $n=0$}; \\
\ttt{1},  & \text{if $n=-1$}; \\
\ttt{1} \cdot |m_{k-1}| \cdot |m_{k-2}| \cdots |m_0|,  & \text{if $n\le -2$};
\end{cases}
\]
where if $n\ge 1$ (resp., $n\le -2$),  $k=k(n)$ is the unique integer and $((m_i,a_i))_{i=0,\dots,k-1}$ is the unique sequence obtained from~\cref{thm: unique admin for right infinite periodic} (resp., \cref{thm: unique admin for left infinite periodic}) applied on the right-infinite periodic point $u|_{\N}=u_0u_1\cdots$ (resp.,  the left-infinite periodic point $u|_{\Z_{<0}}=\cdots u_{-2}u_{-1}$) with period $p$.
This numeration system is called the \emph{Dumont--Thomas complement numeration system associated with $u$}.
\end{definition}

Observe that all representations of non-negative integers start with a letter $\ttt{0}$ while those of negative integers start with a letter $\ttt{1}$.
Again,  we use a special font for the digits in the representations of integers.
Also note that,  if we get rid of the initial letter,  the lengths of all representations are multiples of $p$. Using the tree interpretation, this generalization corresponds to using the tree $\mathcal{T}_{\mu,u_{-1}|u_0}$ to accommodate negative numbers, and using only paths of length one plus a multiple of $p$ to manage periodicity.

\begin{example}
Consider the substitution $\mu\colon a\mapsto abc, b\mapsto c,  c\mapsto ac$ and the two-sided periodic point $u\in\Per_{\Z}(\mu)$ with growing seed $c|a$.
(Note that the period of $u$ is $1$, so it is actually a fixed point.)
The corresponding tree $\mathcal{T}_{\mu,c|a}$ is depicted in~\cref{fig:DTNS-mu-u-tree} while the corresponding Dumont--Thomas numeration system is illustrated in~\cref{fig:DTNS-mu-u-table}.
For instance,  the paths $\ttt{02}$ and $\ttt{002}$ lead to the nodes in orange, which both project onto $n=2$ on the horizontal line below the tree.
To represent this integer $n=2$,  we need to take the path leading to the earliest level having a node projecting onto $n=2$,  so $\rep_u(2)=\ttt{02}$.

\begin{figure}
\centering
\subfloat[\label{fig:DTNS-mu-u-tree}]
{
%%%
\begin{tikzpicture}[xscale=.75, yscale=1.4, >=latex,
node0/.style={},
nodeM/.style={fill=cyan!50},
nodeA/.style={fill=orange!50}]
    % vertical dash line
    \draw[thick,dashed] (-.5,-2.8) -- (-.5,.5);
    % nodes positive
    \node                 (O)   at (-.5,1.5) {};
    \node[rectangle,draw,nodeM] (S) at (-.5,1) {start};
    \node[rectangle,draw,node0] (0) at (0,0)  {$a$};
    \node[rectangle,draw,node0] (00) at (0,-1) {$a$};
    \node[rectangle,draw,node0] (01) at (1,-1) {$b$};
    \node[rectangle,draw,nodeA] (02) at (2,-1) {$c$};
    \node[rectangle,draw,node0] (000) at (0,-2) {$a$};
    \node[rectangle,draw,node0] (001) at (1,-2) {$b$};
    \node[rectangle,draw,nodeA] (002) at (2,-2) {$c$};
    \node[rectangle,draw,node0] (010) at (3,-2) {$c$};
    \node[rectangle,draw,node0] (020) at (4,-2) {$a$};
    \node[rectangle,draw,node0] (021) at (5,-2) {$c$};
    % nodes negative
    \node[rectangle,draw,node0] (1) at (-1,0)  {$c$};
    \node[rectangle,draw,node0] (10) at (-2,-1) {$a$};
    \node[rectangle,draw,node0] (11) at (-1,-1) {$c$};
    \node[rectangle,draw,node0] (100) at (-5,-2) {$a$};
    \node[rectangle,draw,node0] (101) at (-4,-2) {$b$};
    \node[rectangle,draw,node0] (102) at (-3,-2) {$c$};
    \node[rectangle,draw,node0] (110) at (-2,-2) {$a$};
    \node[rectangle,draw,node0] (111) at (-1,-2) {$c$};
    %\node at (-6,-2.5) {$-6$};
    \node at (-5,-2.5) {$-5$};
    \node at (-4,-2.5) {$-4$};
    \node at (-3,-2.5) {$-3$};
    \node at (-2,-2.5) {$-2$};
    \node at (-1,-2.5) {$-1$};
    \node at (0,-2.5) {$0$};
    \node at (1,-2.5) {$1$};
    \node at (2,-2.5) {$2$};
    \node at (3,-2.5) {$3$};
    \node at (4,-2.5) {$4$};
    \node at (5,-2.5) {$5$};
    % positive edges
    \draw[->] (O) to (S);
    \draw[->] (S)  -- node[fill=white,inner sep=2pt] {\scriptsize \ttt{0}} (0);
    \draw[->] (S)  -- node[fill=white,inner sep=2pt] {\scriptsize \ttt{1}} (1);
    \draw[->] (0)  -- node[fill=white,inner sep=2pt] {\scriptsize \ttt{0}} (00);
    \draw[->] (0)  -- node[fill=white,inner sep=2pt] {\scriptsize \ttt{1}} (01);
    \draw[->] (0)  -- node[fill=white,inner sep=2pt] {\scriptsize \ttt{2}} (02);
    \draw[->] (00) -- node[fill=white,inner sep=2pt] {\scriptsize \ttt{0}} (000);
    \draw[->] (00) -- node[fill=white,inner sep=2pt] {\scriptsize \ttt{1}} (001);
    \draw[->] (00) -- node[fill=white,inner sep=2pt] {\scriptsize \ttt{2}} (002);
    \draw[->] (01) -- node[fill=white,inner sep=2pt] {\scriptsize \ttt{0}} (010);
    \draw[->] (02) -- node[fill=white,inner sep=2pt] {\scriptsize \ttt{0}} (020);
    \draw[->] (02) -- node[fill=white,inner sep=2pt] {\scriptsize \ttt{1}} (021);
    % negative edges
    \draw[->] (1)  -- node[fill=white,inner sep=2pt] {\scriptsize \ttt{0}} (10);
    \draw[->] (1)  -- node[fill=white,inner sep=2pt] {\scriptsize \ttt{1}} (11);
    \draw[->] (10) -- node[fill=white,inner sep=2pt] {\scriptsize \ttt{0}} (100);
    \draw[->] (10) -- node[fill=white,inner sep=2pt] {\scriptsize \ttt{1}} (101);
    \draw[->] (10) -- node[fill=white,inner sep=2pt] {\scriptsize \ttt{2}} (102);
    \draw[->] (11) -- node[fill=white,inner sep=2pt] {\scriptsize \ttt{0}} (110);
    \draw[->] (11) -- node[fill=white,inner sep=2pt] {\scriptsize \ttt{1}} (111);
\end{tikzpicture}
%%%
}
\hspace{1cm}
\subfloat[\label{fig:DTNS-mu-u-table}]
{
%%%
\begin{tabular}{|c|c|r|}
                        $n$ & path & $\rep_u(n)$\\ \hline
                        -5 & \ttt{100}         &  \ttt{100}           \\
                        -4 & \ttt{101}         &  \ttt{101}           \\
                        -3 & \ttt{102}         & \ttt{102}          \\
                        -2 & \ttt{110}         & \ttt{10}         \\
                        -1 & \ttt{111}        & \ttt{1}          \\
                        0  & \ttt{000}         & \ttt{0}          \\
                        1  & \ttt{001}         & \ttt{01}          \\
                        2  & \ttt{002}        & \ttt{02}           \\
                        3  & \ttt{010}         & \ttt{010}          \\
                        4  & \ttt{020}         & \ttt{020}         \\
                        5  & \ttt{021}       & \ttt{021}     \\
        \end{tabular}
%%%
 }
\caption{On the left,  the tree $\mathcal{T}_{\mu,c|a}$ associated with the two-sided periodic point $u$ of the substitution $\mu\colon a\mapsto abc, b\mapsto c,  c\mapsto ac$ with growing seed $c|a$.  On the right,  the representations of the first few integers in the corresponding Dumont--Thomas complement numeration system.}
\label{fig:DTNS-mu-u}
\end{figure}
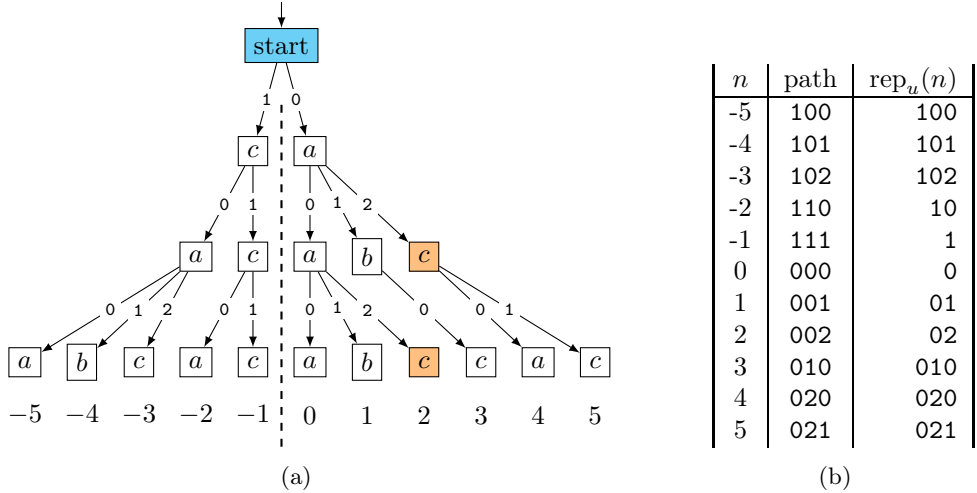
\end{example}

In the next example,  we consider two substitutions associated with the same real number.
The key observation is that one gives rise to a numeration system that is positional,  and the numeration system for the other is not.
This example is already present in Lep\v sov\'a's PhD thesis~\cite{Lepsova-2024}.

\begin{example}
 {\rm\cite[Example 6.5.6]{Lepsova-2024}}
 Consider the \emph{silver mean} $\beta=1+\sqrt{2}$.
 We define the substitutions $\mu\colon a \mapsto aab, b\mapsto a$ and $\rho\colon a\mapsto abb,  b\mapsto ab$.
 The characteristic polynomials of the corresponding adjacency matrices of $\mu$ and $\rho$ are equal and have dominant root $\beta$.
 We now consider the two-sided periodic points $u\in\Per_{\Z}(\mu)$ with growing seed $b|a$ and  period $2$ and $v\in\Per_{\Z}(\rho)$ with growing seed $b|a$ and period $1$.
 The representations of the first few integers in the corresponding Dumont--Thomas numeration systems 
 %with $r=0$ 
 are given in~\cref{tab:DTNS silver mean}.

 On the one hand, the system based on $\rho$ is not positional. Indeed,  if it were,  there would be an evaluation map $\val$ such that $\val(\rep_{v,0}(n))=n$ for all integers $n$.
Since $\rep_{v,0}(3)=\ttt{010}$, we must have $U_1=3$,  but then $\val(\rep_{v,0}(5))=\val(\ttt{020})=6\neq 5$,  a contradiction.

 On the other hand, there is no similar obstruction for the system based on $\mu$. In fact, this system is indeed positional, which we sketch now. Consider first the representation of non-negative numbers. The set $\rep_u(\N)$ is the set of words of odd length, that start with $\ttt{0}$ but not with $\ttt{000}$, and that do not contain the factors $\ttt{21}$ or $\ttt{22}$ (as reading a $\ttt{2}$ puts us at a node $b$, from which the only outgoing edge is labeled $\ttt{0}$). 
 
 If we now consider the positional numeration system $U$ based on the sequence $(U_n)_{n\in \N}$ given by
 \[U_0=1,\, U_1=3 \text{ and }U_{n+2}=2U_{n+1}+U_n,\]
 the set of valid representations for this system is the set of words on $\{\ttt{0},\ttt{1},\ttt{2}\}$ that do not start with $\ttt{0}$ and do not contain any factor $\ttt{21}$ or $\ttt{22}$. Since in both systems the representation map is increasing in the radix order (see \cite[Section 8]{MR4836876}), there is a bijection between the representation of $n$ in both systems, that consists of adding or removing one or two lead zeroes to ensure that the necessary representations are of odd length. This bijection preserves positionality, with the weights unchanged between the two systems. 
 
  For the representation of negative numbers, we first note that $\ttt{1}(\ttt{00})^n$ is the representation of $-\len{\mu^n(b)}$ (which gives us the weight sequence $V$). Then, the word $10w$, where $w$ is of length $2n-1$ and has no factor $\ttt{21}$ or $\ttt{22}$, has value $-\len{\mu^n(b)}+\val_U(w)$, and positionality is once again preserved.
  
  As for why we chose to introduce this particular system $U$, we can mention that the polynomial of its recurrence is the characteristic polynomial of the adjacency matrix of $\mu$, and it can also be found from the two-state automaton recognizing the set of words without the factors $\ttt{21}$ or $\ttt{22}$.

 \begin{table}
 \centering
 \caption{The representations of the first few integers in the Dumont--Thomas numeration system associated with $u\in\Per_{\Z}(\mu)$ and $v\in\Per_{\Z}(\rho)$ and $r=0$ where  $\mu\colon a \mapsto aab, b\mapsto a$ and $\rho\colon a\mapsto abb,  b\mapsto ab$, both with seed $b|a$.}
        \label{tab:DTNS silver mean}
        \begin{center}
        \begin{tabular}{c|ccccc|cccccc}
 		$n$ & $-5$ & $-4$ & $-3$ & $-2$ & $-1$ & $0$ & $1$ & $2$ & $3$ & $4$ & $5$ \\
 		\hline
 		$\rep_u(n)$ & $\ttt{10112}$ & $\ttt{10120}$ & $\ttt{100}$ & $\ttt{101}$  & $\ttt{1}$ & $\ttt{0}$ & $\ttt{001}$ & $\ttt{002}$ & $\ttt{010}$ & $\ttt{011}$ & $\ttt{012}$ \\
 		\hline
 		$\rep_v(n)$ & $\ttt{100}$ & $\ttt{101}$ & $\ttt{102}$ & $\ttt{10}$  & $\ttt{1}$ & $\ttt{0}$ & $\ttt{01}$ & $\ttt{02}$ & $\ttt{010}$ & $\ttt{011}$ & $\ttt{020}$
        \end{tabular}
        \end{center}
       \end{table}
\end{example}

As suggested by the previous example,  Dumont--Thomas numeration systems for $\Z$ corresponding to periodic points of substitution are not always positional. 
Therefore,  Lep\v sov\'a raised the following natural question~\cite{Lepsova-2024}.

\questionB
%\begin{question}
%\label{quest:DT case}
%{\rm\cite[Question 6.5.7]{Lepsova-2024}}
%What are the conditions for a complement Dumont--Thomas numeration system to be positional?
%\end{question}

\subsection{Changing the lengths in the numeration language}
\label{sec:DT for N and Z with various residues}

As already observed,  \cref{thm: unique admin for right infinite periodic,thm: unique admin for left infinite periodic} both provide a numeration language whose words are of length multiple of the period $p$  of the considered infinite word. 
In this section,  we show that this is a superfluous restriction and we show that all lengths can be attained modulo $p$, a fact which is clear when considering the tree interpretation.
%\todoS{Maybe explain this interpretation a bit more explicitly?}
We start with right-infinite periodic points.

\begin{theorem}
\label{thm: unique admin for right infinite periodic len r}
Let $\mu:A^*\to A^*$ be a substitution with growing letter $a\in A$.
Consider a right-infinite periodic point $u\in \Per_{\N}(\mu)$ with $u_0=a$ and period $p\ge 1$.
Fix a residue $r\in\{0,1,\ldots,p-1\}$ modulo $p$ and define $v_r = \mu^r(u)$.
For every integer $n\geq 0$, there exist a unique integer $k=k(n)$ with $k\equiv r \bmod p$ and a unique sequence $((m_i,a_i))_{i=0,\dots,k-1}$ such that the sequence is $a$-admissible,  
\begin{equation}
\label{eq: non equality with epsilon}
m_{k-1}m_{k-2} \cdots m_{k-p}\neq\varepsilon \text{ if }k\geq p,  
\end{equation}
and $ (v_r)_{[0,n-1]} = \mu^{k-1}(m_{k-1})\mu^{k-2}(m_{k-2}) \cdots \mu^0(m_0)$.
\end{theorem}

The proof of~\cref{thm: unique admin for right infinite periodic len r} mimics that of~\cref{thm: unique admin for right infinite periodic} but we choose to present it for the sake of completeness.
Some intermediate results are necessary,  which we recall too.

\begin{lemma}
 {\rm\cite[Lemma 1.1]{Dumont-Thomas-1989}}
 \label{lem: DT bounds for sum of lengths}
 Let $\mu\colon A^* \to A^*$ be a substitution and let $k\ge 0$ be an integer.
 If $((m_i,a_i))_{i=0,\dots,k}$ in an admissible sequence,  then $\sum_{i=0}^k \len{\mu^i(m_i)} < \len{\mu^k(m_ka_k)}$.
 \end{lemma}
 
 \begin{lemma}
 {\rm\cite[Lemma 3.9]{MR4836876}}
 \label{lem: DT decomposition for proper prefix}
 Let $\mu\colon A^* \to A^*$ be a substitution and let $k\ge 0$ be an integer.
 Let also $m\in A^*$ and $a\in A$ be such that $m$ is a proper prefix of $\mu^k(a)$.
 Then there exists a unique $a$-admissible sequence $((m_i,a_i))_{i=0,\dots,k-1}$ such that $m = \mu^{k-1}(m_{k-1})\mu^{k-2}(m_{k-2})\cdots \mu^0(m_0)$.
 \end{lemma}
 
\begin{proof}[Proof of~\cref{thm: unique admin for right infinite periodic len r}]
Since $a$ is a growing letter,  the sequence $(\len{\mu^{\ell p+r}(a)})_{\ell\ge 0}$ is (strictly) increasing.
Note that,  for every $r\in\{0,1,\ldots,p-1\}$,  the word $v_r$ is also a periodic point of $\mu$ with period $p$.
Let $n\ge 0$ be an integer. Setting $\len{\mu^{-p+r}(a)}=0$ by convention,
there exists a unique integer $\ell=\ell(n)\geq 0$ such that $\len{ \mu^{(\ell-1)p+r}(a)} \le n < \len{ \mu^{\ell p+r}(a)}$.
Let $k=k(n)=\ell p + r$ be such that
\begin{equation}\label{eq: squeeze n between two consec lengths}
\len{ \mu^{k-p}(a)} \le n < \len{ \mu^k(a)}.
\end{equation}
Now the word $m = v_{r,0}v_{r,1} \cdots v_{r,n-1}$ of length $n$ is a proper prefix of $\mu^k(a)$ (for a visualization,  recall that the $k$-th level of the tree $\mathcal{T}_{\mu,a}$ is a prefix of $v_r$ since $k\equiv r\bmod p$).
By~\cref{lem: DT decomposition for proper prefix},  there exists a unique $a$-admissible sequence $((m_i,a_i))_{i=0,\dots,k-1}$ such that $m = \mu^{k-1}(m_{k-1})\cdots \mu^0(m_0)$.
It remains to show that $m_{k-1}m_{k-2} \cdots m_{k-p}\neq\varepsilon$ if $k\geq p$.
Towards a contradiction,  we assume that $k\geq p$ and $m_{k-1}m_{k-2} \cdots m_{k-p}=\varepsilon$.
By inspecting the definition of admissible sequences with periodic points,  it means that $a_{k-p}=a$.
Then~\cref{lem: DT bounds for sum of lengths} implies that
\begin{align*}
n 
&= \len{m} = \sum_{j=0}^{k-1} \len{ \mu^j(m_j)} = \sum_{j=0}^{k-p-1} \len{ \mu^j(m_j)} \\
&< \len{ \mu^{k-p-1}(m_{k-p-1}a_{k-p-1})} \le \len{ \mu^{k-p-1}(\mu(a_{k-p}))} = \len{\mu^{k-p}(a)},
\end{align*}
which contradicts~\cref{eq: squeeze n between two consec lengths}.  Consequently,  $m_{k-1}m_{k-2} \cdots m_{k-p}\neq \varepsilon$,  which finishes the proof of existence. 

For uniqueness, we already know that the sequence $((m_i,a_i))_{i=0,\ldots,k-1}$ constructed above is the only $a$-admissible sequence of length $k$ that satisfies the conditions of the statement. It remains to show that other lengths do not lead to more such sequences. First, suppose that $((m'_i,a'_i))_{i=0,\ldots,k'-1}$ is an $a$-admissible sequence of length $k'$ with $k'<k$ that satisfies the conditions of the statement. Since $k'\equiv r\bmod p$, we have $k'\leq k-p$. From~\cref{lem: DT bounds for sum of lengths}, we deduce that 
\[\len{\mu^{k'-1}(m'_{k'-1})\cdots \mu^0(m'_0)}<\len{\mu^{k'-1}(m'_{k'-1}a'_{k'-1})}\leq \len{\mu^{k'}(a)}\leq \len{\mu^{k-p}(a)}\leq n= \len{m}.\]
Thus it must be that $\mu^{k'-1}(m'_{k'-1})\cdots \mu^0(m'_0)\neq m$,  which violates the last condition of the statement.
Now for $k'>k$, we note that the sequence obtained by prefixing 
\[
(m_{k'-1},a_{k'-1})=(\varepsilon,\mu(a)_0),
(m_{k'-2},a_{k'-2})=(\varepsilon,\mu^2(a)_0),\ldots,
(m_k,a_k)=(\varepsilon,\mu^{k'-k}(a)_0)
\] 
to the sequence $((m_i,a_i))_{i=0,\ldots,k-1}$ produces an $a$-admissible sequence of length $k'$ which is such that $m=\mu^{k'-1}(m_{k'-1})\cdots \mu^0(m_0)$. By~\cref{lem: DT decomposition for proper prefix}, this is the only such sequence of length $k'$. But this sequence starts with $m_{k'-1}=\cdots=m_{k'-p}=\varepsilon$, and thus it does not satisfy the statement. Thus the uniqueness is shown.
\end{proof}

We now move to left-infinite periodic points.

\begin{theorem}
\label{thm: unique admin for left infinite periodic len r}
Let $\mu:A^*\to A^*$ be a substitution with growing letter $b\in A$.
Consider a left-infinite periodic point $u\in \Per_{\Z_{<0}}(\mu)$ with $u_{-1}=b$ and period $p\ge 1$.
Fix a residue $r\in\{0,1,\ldots,p-1\}$ modulo $p$ and define $v_r = \mu^r(u)$.
For every integer $n\leq -1$, there exist a unique integer $k=k(n)$ with $k\equiv r \bmod p$ and a unique sequence $((m_i,a_i))_{i=0,\dots,k-1}$ such that 
\begin{enumerate}
\item the sequence $((m_i,a_i))_{i=0,\dots,k-1}$ is $b$-admissible,  
\item we have $-|\mu^k(b)|\leq n$, 
\item we have
\begin{equation}
\label{eq: non equality with mupa}
\mu^{p-1}(m_{k-1})\mu^{p-2}(m_{k-2}) \cdots \mu^0(m_{k-p}) a_{k-p} \neq \mu^p(b) \text{ if }k\geq p,
\end{equation}
\item and we have $(v_r)_{[-|\mu^k(b)|, n-1]} = \mu^{k-1}(m_{k-1})\mu^{k-2}(m_{k-2}) \cdots \mu^0(m_0)$.
\end{enumerate}
\end{theorem}

\begin{proof}
Note that,  for every $r\in\{0,1,\ldots,p-1\}$,  the word $v_r$ is also a periodic point of $\mu$ with period $p$.
Since $b$ is a growing letter,  the sequence $(-\len{\mu^{\ell p+r}(b)})_{\ell\ge 0}$ is (strictly) decreasing.
Let $n\le -1$ be an integer.
Setting $\len{\mu^{-p+r}(b)}=0$ by convention, there exists a unique integer $\ell=\ell(n)$ such that $- \len{ \mu^{\ell p+r}(b)}\le n < -\len{ \mu^{(\ell-1)p+r}(b)}$.
Let $k=k(n)=\ell p + r$ be such that
\begin{equation}\label{eq: squeeze n between two consec lengths left case}
-\len{ \mu^k(b)} \le n < - \len{ \mu^{k-p}(b)}.
\end{equation}
Now the word $m = v_{r,-\len{\mu^k(b)}} \cdots v_{r,n-2}v_{r,n-1}$ of length 
\begin{equation}\label{eq: length of prefix of vr}
\len{m} =  \len{\mu^k(b)} + n <  \len{\mu^k(b)}- \len{ \mu^{k-p}(b)} \le  \len{\mu^k(b)}
\end{equation}
is a proper prefix of $\mu^k(b)$.
Using~\cref{lem: DT decomposition for proper prefix},  there exists a unique $b$-admissible sequence $((m_i,a_i))_{i=0,\dots,k-1}$ such that $m = \mu^{k-1}(m_{k-1})\mu^{k-2}(m_{k-2})\cdots \mu^0(m_0)$.
We now show that~\cref{eq: non equality with mupa} holds.
Towards a contradiction,  we assume that~\cref{eq: non equality with mupa} is actually an equality.
Then we have $a_{k-p}=b$.
Writing
\[
m = \mu^{k-p}(\mu^{p-1}(m_{k-1}) \cdots \mu^0(m_{k-p}) ) \mu^{k-p-1}(m_{k-p-1}) \cdots \mu^0(m_0),
\]
we obtain from~\cref{eq: non equality with mupa} that
\[
\len{m} 
= \len{\mu^{k-p}(\mu^p(b))} - \len{\mu^{k-p}(a_{k-p})} +  \sum_{j=0}^{k-p-1} \len{ \mu^j(m_j)} 
\ge \len{\mu^k(b)} - \len{\mu^{k-p}(b)},
\]
which contradicts~\cref{eq: length of prefix of vr}.  Consequently,   \cref{eq: non equality with mupa} is a non-equality,  which finishes the proof of existence.

For uniqueness, we already know that the sequence $((m_i,a_i))_{i=0,\ldots,k-1}$ constructed above is the only $b$-admissible sequence of length $k$ that satisfies the conditions of the statement. It remains to show that other lengths do not lead to other sequences of this kind.
For $k'<k$, we clearly do not have the second condition $-|\mu^k(b)|\leq n$ of the statement. On the other hand, 
for $k'>k$, we note that when $k\equiv r \bmod p$, prefixing 
\begin{align*}
(m_{k'-1},a_{k'-1})&=(\mu(b)_{[0,\len{\mu(b)}-1)},\mu(b)_{\len{\mu(b)}-1}),\\
(m_{k'-2},a_{k'-2})&=(\mu(a_{k'-1})_{[0,\len{\mu(a_{k'-1})}-1)},\mu(a_{k'-1})_{\len{\mu(a_{k'-1})}-1}),\\
&\vdots\\
(m_{k},a_{k})&=(\mu(a_{k+1})_{[0,\len{\mu(a_{k+1})}-1)},\mu(a_{k+1})_{\len{\mu(a_{k+1})}-1})
\end{align*}
to the sequence $((m_i,a_i))_{i=0,\ldots,k-1}$ produces a $b$-admissible sequence of length $k'$ which is such that $v_{[-\len{\mu^{k'}(b)},n-1]}=\mu^{k'-1}(m_{k'-1})\cdots \mu^0(m_0)$.
By~\cref{lem: DT decomposition for proper prefix}, this is the only such sequence of length $k'$. But this sequence violates~\cref{eq: non equality with mupa}.
\end{proof}

From these results,  we generalize the definition of Dumont--Thomas numeration systems~\cite{Dumont-Thomas-1989,MR4836876}.

\begin{definition}
\label{def: DTNS for Z equiv r mod p}
Let $\mu \colon A^* \to A^*$ be a substitution and let $u\in\Per_{\Z}(\mu)$ be a two-sided periodic point with growing seed $u_{-1}|u_0$ and period $p\ge 1$. 
Let $r\in\{0,1,\ldots,p-1\}$ be a residue modulo $p$.
Define $\ttt{c}=\max_{a\in A}|\mu(a)|-1$  and the set $D=\{\mtt{0},\mtt{1},\ldots,\ttt{c}\}$.
We define the map $\rep_{u,r} \colon \Z \to \{\mtt{0},\mtt{1}\}D^*, n \mapsto \rep_{u,r} (n)$ by
\[
\rep_{u,r} (n) = 
\begin{cases}
\mtt{0} \cdot |m_{k-1}| \cdot |m_{k-2}| \cdots |m_0|,  & \text{if $n\ge 0$};\\
\mtt{1} \cdot |m_{k-1}| \cdot |m_{k-2}| \cdots |m_0|,  & \text{if $n\le -1$};
\end{cases}
\]
where $k=k(n)$ is the unique integer congruent to $r$ modulo $p$ and $((m_i,a_i))_{i=0,\dots,k-1}$ is the unique sequence obtained from~\cref{thm: unique admin for right infinite periodic len r} (resp., \cref{thm: unique admin for left infinite periodic len r}) applied on the right-infinite periodic point $u|_{\N}=u_0u_1\cdots$ (resp.,  the left-infinite periodic point $u|_{\Z_{<0}}=\cdots u_{-2}u_{-1}$) with period $p$.

This numeration system is called the \emph{Dumont--Thomas complement numeration system associated with $\mu$, $u$ and $r$}.
When the context is clear,  we drop the dependence on $\mu$, $u$ and $r$.
\end{definition}

Note that the lengths of the representations in the previous numeration system are congruent to $r+1$ modulo $p$.
Observe that all representations of non-negative integers start with a letter $\ttt{0}$ while those of negative integers start with a letter $\ttt{1}$.
Note that we again use a special font for the digits representing numbers to distinguish them from integers.

\begin{example}
\label{ex: phi and phi-squared}
Consider the substitution $\mu \colon a\mapsto ccd,  b \mapsto cd,  c \mapsto ab,  d\mapsto a$.
Take the periodic point $u\in\Per_{\N}(\mu)$ with growing seed $a|a$ and period $p=2$.
The tree $\mathcal{T}_{\mu,a|a}$ is depicted on~\cref{fig: tree phi and phi-squared}.
Now depending on whether we want representations of even or odd lengths,  we obtain different numeration systems as illustrated on the table in~\cref{fig: tree phi and phi-squared}.
\begin{figure}[h]
\centering
\subfloat
{
%%%
\begin{tikzpicture}[xscale=.75, yscale=1.4, >=latex,
node0/.style={},
nodeM/.style={fill=cyan!50},
nodeA/.style={fill=orange!50}]
\node[rectangle,draw,nodeM] (S) at (-.5,1) {start};
%positive nodes
    \node[rectangle,draw,node0] (0) at (0,0)  {$a$};
    \node[rectangle,draw,node0] (00) at (0,-1) {$c$};
    \node[rectangle,draw,node0] (01) at (1,-1) {$c$};
    \node[rectangle,draw,node0] (02) at (2,-1) {$d$};
    \node[rectangle,draw,node0] (000) at (0,-2) {$a$};
    \node[rectangle,draw,node0] (001) at (1,-2) {$b$};
    \node[rectangle,draw,node0] (010) at (2,-2) {$a$};
    \node[rectangle,draw,node0] (011) at (3,-2) {$b$};
    \node[rectangle,draw,node0] (020) at (4,-2) {$a$};
    \node[rectangle,draw,node0] (0000) at (0,-3) {$c$};
    \node[rectangle,draw,node0] (0001) at (1,-3) {$c$};
    \node[rectangle,draw,node0] (0002) at (2,-3) {$d$};
    \node[rectangle,draw,node0] (0010) at (3,-3) {$c$};
    \node[rectangle,draw,node0] (0011) at (4,-3) {$d$};
    \node[rectangle,draw,node0] (0100) at (5,-3) {$c$};
    \node[rectangle,draw,node0] (0101) at (6,-3) {$c$};
    \node[rectangle,draw,node0] (0102) at (7,-3) {$d$};
    \node[rectangle,draw,node0] (0110) at (8,-3) {$c$};
    \node[node0] (0111) at (9,-3) {$\cdots$};    
    \node[] (zer) at (0,-3.5)  {$0$};
    \node[] (un) at (1,-3.5)  {$1$};
    \node[] (2b) at (2,-3.5)  {$2$};
    \node[] (3) at (3,-3.5)  {$3$};
    \node[] (4) at (4,-3.5)  {$4$};
    \node[] (5) at (5,-3.5)  {$5$};
    \node[] (6) at (6,-3.5)  {$6$};
    \node[] (7) at (7,-3.5)  {$7$};
    \node[] (8) at (8,-3.5)  {$8$};
%negative nodes
    \node[rectangle,draw,node0] (1) at (-1,0)  {$a$};
    \node[rectangle,draw,node0] (10) at (-3,-1)  {$c$};
    \node[rectangle,draw,node0] (11) at (-2,-1)  {$c$};
    \node[rectangle,draw,node0] (12) at (-1,-1)  {$d$};
    \node[rectangle,draw,node0] (100) at (-5,-2)  {$a$};
    \node[rectangle,draw,node0] (101) at (-4,-2)  {$b$};
    \node[rectangle,draw,node0] (110) at (-3,-2)  {$a$};
    \node[rectangle,draw,node0] (111) at (-2,-2)  {$b$};
    \node[rectangle,draw,node0] (120) at (-1,-2)  {$a$};
    \node[rectangle,draw,node0] (1110) at (-5,-3)  {$c$};
    \node[rectangle,draw,node0] (1111) at (-4,-3)  {$d$};
    \node[rectangle,draw,node0] (1200) at (-3,-3)  {$c$};
    \node[rectangle,draw,node0] (1201) at (-2,-3)  {$c$};
    \node[rectangle,draw,node0] (1202) at (-1,-3)  {$d$};
    \node[node0] (mcdots) at (-6,-3) {$\cdots$};    
    \node[] (m1) at (-1,-3.5)  {$-1$};
    \node[] (m2) at (-2,-3.5)  {$-2$};
    \node[] (m3) at (-3,-3.5)  {$-3$};
    \node[] (m4) at (-4,-3.5)  {$-4$};
    \node[] (m5) at (-5,-3.5)  {$-5$};
    %positive edges
    \draw[->] (S)  -- node[fill=white,inner sep=2pt] {\scriptsize \ttt{0}} (0);
    \draw[->] (0)  -- node[fill=white,inner sep=2pt] {\scriptsize \ttt{0}} (00);
    \draw[->] (0)  -- node[fill=white,inner sep=2pt] {\scriptsize \ttt{1}} (01);
    \draw[->] (0)  -- node[fill=white,inner sep=2pt] {\scriptsize \ttt{2}} (02);
    \draw[->] (00) -- node[fill=white,inner sep=2pt] {\scriptsize \ttt{0}} (000);
    \draw[->] (00) -- node[fill=white,inner sep=2pt] {\scriptsize \ttt{1}} (001);
    \draw[->] (01) -- node[fill=white,inner sep=2pt] {\scriptsize \ttt{0}} (010);
    \draw[->] (01) -- node[fill=white,inner sep=2pt] {\scriptsize \ttt{1}} (011);
    \draw[->] (02) -- node[fill=white,inner sep=2pt] {\scriptsize \ttt{0}} (020);
     \draw[->] (000) -- node[fill=white,inner sep=2pt] {\scriptsize \ttt{0}} (0000);
     \draw[->] (000) -- node[fill=white,inner sep=2pt] {\scriptsize \ttt{1}} (0001);
     \draw[->] (000) -- node[fill=white,inner sep=2pt] {\scriptsize \ttt{2}} (0002);
     \draw[->] (001.south east) -- node[fill=white,inner sep=2pt] {\scriptsize \ttt{0}} (0010);
     \draw[->] (001.south east) -- node[fill=white,inner sep=2pt] {\scriptsize \ttt{1}} (0011);
     \draw[->] (010.south east) -- node[fill=white,inner sep=2pt] {\scriptsize \ttt{0}} (0100.north west);
     \draw[->] (010.south east) -- node[fill=white,inner sep=2pt] {\scriptsize \ttt{1}} (0101.north west);
     \draw[->] (010.south east) -- node[fill=white,inner sep=2pt] {\scriptsize \ttt{2}} (0102.north west);
     \draw[->] (011.south east) -- node[fill=white,inner sep=2pt] {\scriptsize \ttt{0}} (0110.north west);
     \draw[->] (011.south east) -- node[fill=white,inner sep=2pt] {\scriptsize \ttt{1}} (0111.north);
     %negative edges  
     \draw[->] (S)  -- node[fill=white,inner sep=2pt] {\scriptsize \ttt{1}} (1);
    \draw[->] (1)  -- node[fill=white,inner sep=2pt] {\scriptsize \ttt{0}} (10);
    \draw[->] (1)  -- node[fill=white,inner sep=2pt] {\scriptsize \ttt{1}} (11);
    \draw[->] (1)  -- node[fill=white,inner sep=2pt] {\scriptsize \ttt{2}} (12);
    \draw[->] (10)  -- node[fill=white,inner sep=2pt] {\scriptsize \ttt{0}} (100);
    \draw[->] (10)  -- node[fill=white,inner sep=2pt] {\scriptsize \ttt{1}} (101);
    \draw[->] (11) -- node[fill=white,inner sep=2pt] {\scriptsize \ttt{0}} (110);  
    \draw[->] (11) -- node[fill=white,inner sep=2pt] {\scriptsize \ttt{1}} (111); 
    \draw[->] (12) -- node[fill=white,inner sep=2pt] {\scriptsize \ttt{0}} (120); 
    \draw[->] (111.south west) -- node[fill=white,inner sep=2pt] {\scriptsize \ttt{0}} (1110.north east);
    \draw[->] (111.south west) -- node[fill=white,inner sep=2pt] {\scriptsize \ttt{1}} (1111);  
    \draw[->] (120) -- node[fill=white,inner sep=2pt] {\scriptsize \ttt{0}} (1200); 
    \draw[->] (120) -- node[fill=white,inner sep=2pt] {\scriptsize \ttt{1}} (1201); 
    \draw[->] (120) -- node[fill=white,inner sep=2pt] {\scriptsize \ttt{2}} (1202);
\end{tikzpicture}
%%%
}
\hspace{0.2cm}
\subfloat
{
%%%
\begin{tabular}{c@{\hspace{3pt}}|c ccc ccc ccc cc}
                        $n$ & -4 & -3 & -2 & -1 & 0 & 1 & 2 & 3 & 4 & 5 & 6 & 7\\
                        \hline 
                        $\rep_{u,0}(n)$  & \ttt{101} &  \ttt{110} & \ttt{111} & \ttt{1} & \ttt{0} & \ttt{001} &  \ttt{010} & \ttt{011} & \ttt{020} & \ttt{00100} & \ttt{00101} & \ttt{00110}  \\
                        \hline
                        $\rep_{u,1}(n)$ & \ttt{1111} & \ttt{10} & \ttt{11} & \ttt{12} & \ttt{00} & \ttt{01} & \ttt{02} & \ttt{0010} & \ttt{0011} & \ttt{0100} & \ttt{0101} & \ttt{0102} 
\end{tabular}
%%%
 }
\caption{On the top,  the tree $\mathcal{T}_{\mu,a|a}$ for the substitution $\mu \colon a\mapsto ccd,  b \mapsto cd,  c \mapsto ab,  d\mapsto a$ and the periodic point $u$ of period $p=2$ and seed $a|a$.  On the bottom,  depending on the residue $r\in\{0,1\}$,  we obtain a Dumont--Thomas numeration system and we give $(\rep_{u,r}(n))_{-4\le n\le 7}$ whose lengths are congruent to $r+1 \bmod p$.}
\label{fig: tree phi and phi-squared}
\end{figure}
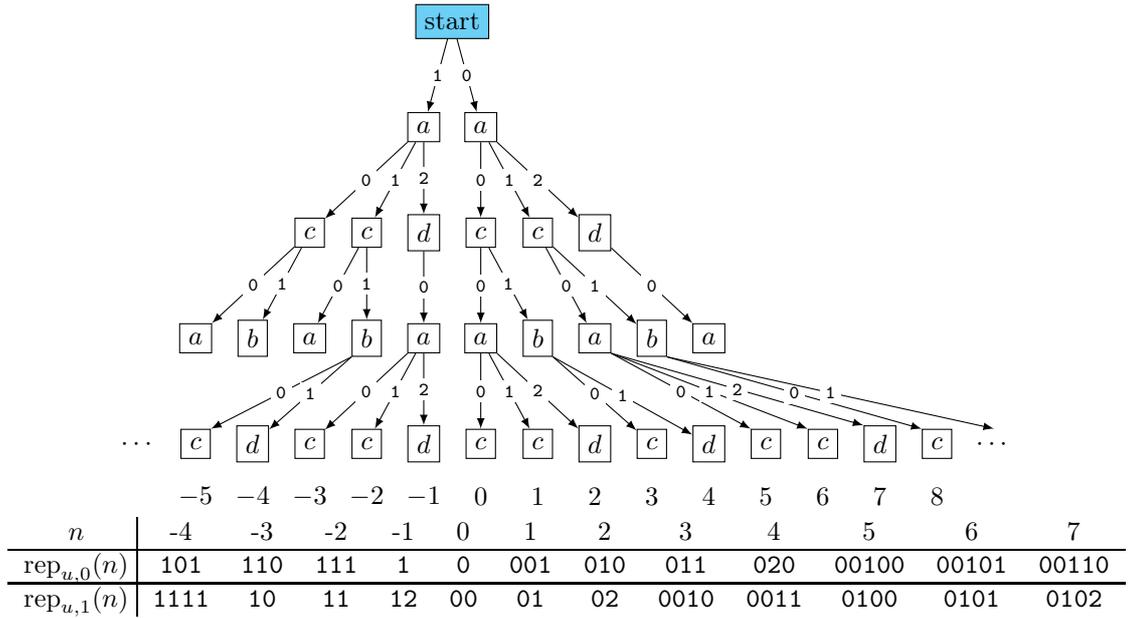
\end{example}

\section{Positional Dumont--Thomas numeration systems}
\label{sec: positionality of DT}

In this section, we study when a substitution generates a Dumont--Thomas numeration system that is positional to answer Question~\ref{quest:DT case}.  We state our theorem in the most general case, then present in~\cref{sec:particularizations} some particularizations as corollaries. 

\subsection{Sketch of the argument}
\label{sec: sketch}

As we will see in~\cref{sec:mainresult}, our main result that answers Question~\ref{quest:DT case} is quite technical. 
The aim of this section is to informally sketch the core argument, then present examples where this argument fails. Thus we justify the need for the technicalities of~\cref{sec:mainresult}, while explaining the intuitive reasoning behind the main result without them getting in the way of the explanation.

\begin{sketch}\label{ske:sketch-main}
We let $\mu \colon A^* \to A^*$ be a substitution,  $u\in\Per_{\Z}(\mu)$ be a two-sided periodic point of $\mu$ with growing seed $b|a$ and period $p\ge 1$,  and $r\in\{0,1,\ldots,p-1\}$ be a residue.
We also consider the corresponding Dumont--Thomas complement numeration system associated with $\mu$,  $u$ and $r$.

Now consider two words $wt0^\ell$ and $w(t+1)0^\ell$ that label two paths in the tree $\mathcal{T}_{\mu,b|a}$ (see~\cref{fig:sketch-main}, where we have arbitrarily chosen to represent the situation with $\mathcal{T}_{\mu,a}$ instead) and assume these words are the representations of the integers $n_1$ and $n_2$ respectively,  i.e., the paths end on columns $n_1$ and $n_2$,  in addition to which we have some conditions for admissibility.
We also let $c$ and $d$ be the letters attained after reading $wt$ and $w(t+1)$ respectively.

If the numeration system is positional, then $n_2-n_1$ must equal $U_\ell$,  the weight in position $\ell$, as the representations of $n_1$ and $n_2$ differ only by one unit in position $\ell$. 
However, $n_2-n_1$ is the number of columns between those two nodes in the tree, which is $\len{\mu^\ell(c)}$, i.e.,  the number of level-$\ell$ descendants of $c$ in the tree (see again~\cref{fig:sketch-main}). Since this reasoning could be applied to any letter $c$ that has a sibling to its right in $\mathcal{T}_{\mu,b|a}$,  this  points towards the following implication: \emph{If the Dumont--Thomas numeration system associated with $\mu$,  $u$,  and $r$ is positional, then $\len{\mu^\ell(c)}=U_\ell$ for every letter $c$ that has a younger sibling in $\mathcal{T}_{\mu,b|a}$.}

The converse implication is also justified, almost by definition of a Dumont--Thomas numeration system. If the positive integer $n\ge 1$ is represented by the word $\texttt{0} \cdot w_{k-1}\cdots w_0$, this means that there is an $a$-admissible sequence $((m_i,a_i))_{i=0,\ldots,k-1}$ such that 
\begin{align}
\label{eq: equality for sketch}
u_{[0,n-1]}=\mu^{k-1}(m_{k-1})\ldots \mu^0(m_0)
\end{align}
and $\len{m_i}=w_i$ for every $i\in\{0,\ldots,k-1\}$. 
Because all letters in $m_i$ have $a_i$ as a younger sibling,  their image by $\mu^\ell$ has length $U_\ell$ by assumption. 
Taking the length in~\cref{eq: equality for sketch} yields $n=\sum_{i=0}^{k-1} U_i w_i$, which corresponds to the numeration system being positional with weights $(U_i)_{i\ge 0}$. The case of negative numbers is similar, with one correcting term corresponding to the value of $V_{k-1}$.
\end{sketch}

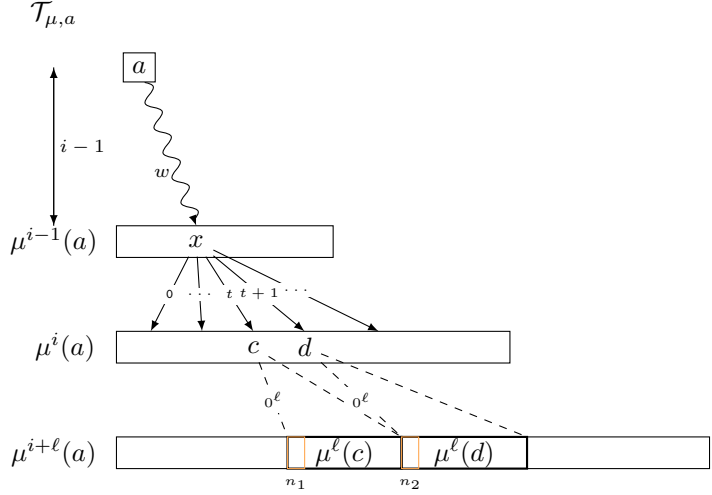
\begin{figure}
\centering
\begin{tikzpicture}
[xscale=.75, yscale=1.4, >=latex,
node0/.style={},
nodeM/.style={fill=cyan!50},
nodeA/.style={fill=orange!50}]
  \node[] (im1) at (-1.5,1)  {$\mathcal{T}_{\mu,a}$};
%vertical
	\node[] (im1) at (-1.5,-1.15)  {$\mu^{i-1}(a)$};
    \node[] (im2) at (-1.3,-2.15)  {$\mu^{i}(a)$};
    \node[] (im2) at (-1.5,-3.15)  {$\mu^{i+\ell}(a)$};
	\draw[->] (-1.5,0.5) -- (-1.5,-1);
	\draw[->] (-1.5,-1) -- (-1.5,0.5);
	\node[] () at (-1,-0.25) {\scriptsize{$i-1$}};
%level 0
    \node[rectangle,draw,node0] (root) at (0,0.5)  {$a$};
%level 1
	\draw [draw=black] (-0.4,-1.3) rectangle (3.4,-1);
	\node[] (c) at (1,-1.15) {$x$};
	\draw[->,decorate,decoration={coil,aspect=0}]   (root)  -- (1,-1);  
	\node[] () at (0.4,-0.5) {\scriptsize{$w$}};
%level 2
    \draw [draw=black] (-0.4,-2.3) rectangle (6.5,-2);
    \node[] (b) at (2,-2.15) {$c$};
    \node[] (d) at (2.9,-2.15) {$d$};
    \draw[->] (c) -- node[pos=0.7, fill=white,inner sep=2pt] {\tiny \ttt{0}} (0.2,-2);
    \draw[->] (c) -- node[pos=0.7, fill=white,inner sep=2pt] {\tiny $\ldots$} (1.1,-2);
    \draw[->] (c) -- node[pos=0.7, fill=white,inner sep=2pt] {\tiny $\ldots$} (4.2,-2);
    \draw[->] (c) -- node[pos=0.7, fill=white,inner sep=2pt] {\tiny \ttt{$t$}} (2,-2);
    \draw[->] (c) -- node[pos=0.7, fill=white,inner sep=2pt] {\tiny \ttt{$t+1$}} (2.9,-2);
%level 3
    \draw [draw=black] (-0.4,-3.3) rectangle (10,-3);
    \draw[dashed]   (b)  -- node[near start, fill=white,inner sep=2pt] {\tiny $0^\ell$} (2.6,-3);  
    \draw[dashed]   (b)  -- (4.6,-3);  
     \draw [thick,draw=black] (2.6,-3.3) rectangle (4.6,-3);
     \node[] (b) at (3.6,-3.15) {$\mu^\ell(c)$};
     \draw [thin,draw=orange!70] (2.6,-3.3) rectangle (2.9,-3);
     \node[] (b) at (2.75,-3.45) {\tiny $n_1$};
     \draw[dashed]   (d)  -- node[near start, fill=white,inner sep=2pt] {\tiny $0^\ell$} (4.6,-3);  
     \draw[dashed]   (d)  -- (6.8,-3);  
     \draw [thick,draw=black] (4.6,-3.3) rectangle (6.8,-3);
     \node[] (b) at (5.7,-3.15) {$\mu^\ell(d)$};
     \draw [thin,draw=orange!70] (4.6,-3.3) rectangle (4.9,-3);
     \node[] (b) at (4.75,-3.45) {\tiny $n_2$};
\end{tikzpicture}
\caption{Comparing the values of $wt0^\ell$ and $w(t+1)0^\ell$ in the right part of $\mathcal{T}_{\mu,b|a}$.}
\label{fig:sketch-main}
\end{figure}

The above sketch leads us to formulate the next conjecture: \emph{The Dumont--Thomas numeration system associated with $\mu$,  $u$, and $r$ is positional if and only if $c \mapsto \len{\mu^\ell(c)}$ is constant on all letters $c$ that have a younger sibling in $\mathcal{T}_{\mu,b|a}$, in which case this constant is the weight $U_\ell$.}

However, trying to prove this conjecture reveals two issues with Sketch~\ref{ske:sketch-main}, which the following examples highlight.

\begin{example}
\label{ex: phi and phi-squared suite}
Recall the substitution $\mu \colon a\mapsto ccd,  b \mapsto cd,  c \mapsto ab,  d\mapsto a$ from~\cref{ex: phi and phi-squared}.
Observe that $\mu$ is not primitive and is built by intertwining on distinct alphabets the Fibonacci substitution $x\mapsto xy,  y\mapsto x$ and the substitution $x\mapsto xxy, y\mapsto xy$ associated with the squared Golden Ratio.
Furthermore,  $\mu$ has periodic points with period $p=2$.

The letters that have a younger sibling in $\mathcal{T}_{\mu,a}$ are $a$ and $c$, but the sequences of the lengths of their consecutive images under $\mu$ are respectively $(\len{\mu^j(a)})_{j\ge 0} = 1, 3, 5, 13, 21, 55, 89, 233, 377,  \ldots$ and $(\len{\mu^j(c)})_{j\ge 0} = 1, 2, 5, 8, 21, 34, 89, 144, 377, \ldots$.
Our tentative condition for positionality is not satisfied. Despite this, the corresponding Dumont--Thomas numeration system is positional for both values of $r$, with the weight sequence $1, 2, 5, 8, 21, 34, \ldots$ for $r=0$ and $1,3,5,13,21,55,\ldots$ for $r=1$, both of which are obtained by taking two out of every three terms in the Fibonacci sequence. 

The way to see this is to examine our proof for the converse implication in the above sketch. With the same notation, for $r=0$ we get that $m_i$ is a power of $a$ if $i$ is even and a power of $c$ if $i$ is odd. Therefore, we obtain the sequence of weights by choosing $U_i=\len{\mu^i(a)}$ if $i$ is even, and  $U_i=\len{\mu^i(c))}$ if $i$ is odd. Since the letters $a$ and $c$ are never present at the same level of the tree, the difference in image length is not a problem in our case.
\end{example}

\cref{ex: phi and phi-squared suite} illustrates that, in the case where $\mu$ is not primitive, not all letters may appear at every level of the tree, and as such we may only control some of their image lengths. This will be the purpose of the sets $E_j$ in the following section with~\cref{def:E_j}.

\begin{example}
\label{ex:bca-bb-b}

Consider the substitution $\mu\colon a\mapsto bca,\, b\mapsto bb,\, c\mapsto b$ and the seed $a|b$. The initial fragment of the associated numeration system over $\Z^{<0}$ is represented in~\cref{fig:ex-bca-bb-b}. The letters $b$ and $c$ both appear in the tree with a younger sibling, and they have images of different lengths. However, the numeration system is still positional. One can show that the weights are $U_i=2^i$ and $V_0=1$,  $V_i=3\cdot 2^{i-1}$ for every $i\geq 1$ by showing that the language of this numeration system is $\texttt{1}\{\texttt{0},\texttt{1}\}^*\setminus \texttt{111}\{\texttt{0},\texttt{1}\}^*$.

Our sketched argument fails, because if $wt$ is a path to a node labeled by $c$ in the tree, then $w(t+1)0^\ell$ is never the representation of any number, due to Condition~\cref{eq: non equality with mupa} from~\cref{thm: unique admin for left infinite periodic len r}. Thus we cannot constrain the lengths of the images of $c$.
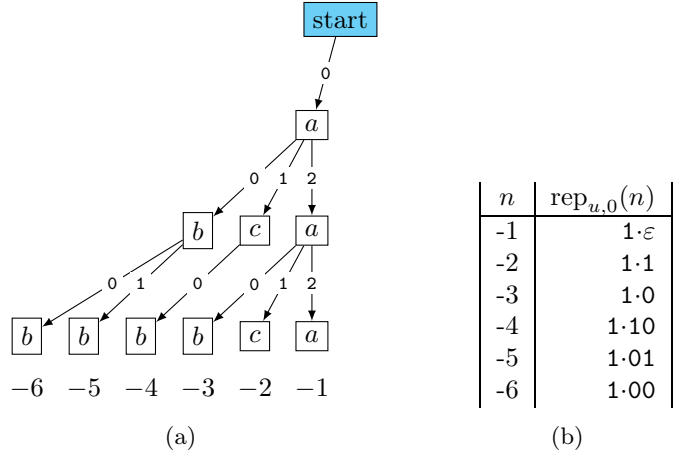
\begin{figure}
\centering
\subfloat[\label{fig:ex-bca-bb-b-tree}]
{
%%%
\begin{tikzpicture}[xscale=.75, yscale=1.4, >=latex,
node0/.style={},
nodeM/.style={fill=cyan!50}]
\node[rectangle,draw,nodeM] (S) at (.5,1) {start};
% nodes
    \node[rectangle,draw,node0] (var) at (0,0)  {$a$};
    \node[rectangle,draw,node0] (0) at (-2,-1) {$b$};
    \node[rectangle,draw,node0] (1) at (-1,-1) {$c$};
    \node[rectangle,draw,node0] (2) at (0,-1) {$a$};
    \node[rectangle,draw,node0] (00) at (-5,-2) {$b$};
    \node[rectangle,draw,node0] (01) at (-4,-2) {$b$};
    \node[rectangle,draw,node0] (10) at (-3,-2) {$b$};
    \node[rectangle,draw,node0] (20) at (-2,-2) {$b$};
    \node[rectangle,draw,node0] (21) at (-1,-2) {$c$};
    \node[rectangle,draw,node0] (22) at (0,-2) {$a$};
    \node[] (m1) at (0,-2.5)  {$-1$};
    \node[] (m2) at (-1,-2.5)  {$-2$};
    \node[] (m3) at (-2,-2.5)  {$-3$};
    \node[] (m4) at (-3,-2.5)  {$-4$};
    \node[] (m5) at (-4,-2.5)  {$-5$};
    \node[] (m5) at (-5,-2.5)  {$-6$};
% edges
	\draw[->] (S)  -- node[fill=white,inner sep=2pt] {\scriptsize \ttt{1}} (var);
    \draw[->] (var)  -- node[fill=white,inner sep=2pt] {\scriptsize \ttt{0}} (0);
    \draw[->] (var)  -- node[fill=white,inner sep=2pt] {\scriptsize \ttt{1}} (1);
    \draw[->] (var)  -- node[fill=white,inner sep=2pt] {\scriptsize \ttt{2}} (2);
    \draw[->] (0) -- node[fill=white,inner sep=2pt] {\scriptsize \ttt{0}} (00);
    \draw[->] (0) -- node[fill=white,inner sep=2pt] {\scriptsize \ttt{1}} (01);
    \draw[->] (1) -- node[fill=white,inner sep=2pt] {\scriptsize \ttt{0}} (10);
    \draw[->] (2) -- node[fill=white,inner sep=2pt] {\scriptsize \ttt{0}} (20);
    \draw[->] (2) -- node[fill=white,inner sep=2pt] {\scriptsize \ttt{1}} (21);
    \draw[->] (2) -- node[fill=white,inner sep=2pt] {\scriptsize \ttt{2}} (22);
\end{tikzpicture}
%%%
}
\hspace{1cm}
\subfloat[\label{fig:ex-bca-bb-b-table}]
{
%%%
\begin{tabular}{|c|r|}
                        $n$ & $\rep_{u,0}(n)$          \\ 
                        \hline
                        -1 & \ttt{1}$\cdot\varepsilon$ \\
                        -2 & \ttt{1$\cdot$1}          \\
                        -3 & \ttt{1$\cdot$0}          \\
                        -4 & \ttt{1$\cdot$10}          \\
                        -5 & \ttt{1$\cdot$01}                   \\
                        -6 & \ttt{1$\cdot$00}                  
\end{tabular}
%%%
}
\caption{The tree $\mathcal{T}_{\mu,a|\cdot}$ and the first few representations of negative integers in the numeration system associated with the substitution $\mu\colon a \mapsto bca,  b \mapsto bb,  c \mapsto b$,  the left-infinite periodic point $u=\cdots bbca$ and residue $r=0$.}
\label{fig:ex-bca-bb-b}
\end{figure}
\end{example}

The sort of argument of~\cref{ex:bca-bb-b} can happen for letters that only appear in column $-2$ in $\mathcal{T}_{\mu,b|a}$ as $w(t+1)$ will lead to a node in column $-1$ and paths extending this one may not correspond to representations of numbers because of Condition~\cref{eq: non equality with mupa}. In fact, such letters may lead to even stranger behavior.

\begin{example}
\label{ex:cursed}
Consider the eight-letter substitution $\mu$ defined by $a_1 \mapsto bca_2$,  $f \mapsto b^2$,  $a_2 \mapsto a_3$, $b \mapsto d^2$,  $c \mapsto d^2e$,  $a_3 \mapsto a_1$,  $d \mapsto f^2$,  $e \mapsto f^4$,  
and its periodic point $u$ with seed $a_1|\cdot$ with period $p=3$ (specifying the right part is not of importance for the sequel).
When considering the residue class $r=2$, this numeration system is not positional as we have $\rep_{u,r}(-6)=\ttt{100}$, $\rep_{u,r}(-4)=\ttt{110}$ and $\rep_{u,r}(-1)=\ttt{120}$, leading to both $U_1=2$ and $U_1=3$, a contradiction. However, if we change $\mu(c)$ to $de$ instead of $d^2e$, the  numeration system becomes positional for $r=2$ despite the fact that the lengths of most images under $\mu$ of $b$ and $c$ are still different.
\end{example}

These examples motivate the detours and technical details present in the next section, notably~\cref{def:C-2}. Keeping these hurdles in mind, we now move on to proving the main result (\cref{thm: main}), using the sketch above but proceeding in a more careful manner.

\subsection{Main result}\label{sec:mainresult}

To state the main result, we need to define some particular sets of letters. 
The setting of this section is as follows: we consider a substitution $\mu \colon A^* \to A^*$  and a two-sided periodic point $u\in\Per_{\Z}(\mu)$ of $\mu$ with growing seed $b|a$ and period $p\ge 1$. 
We draw the tree $\mathcal{T}_{\mu,b|a}$.
For a fixed residue $r\in\{0,\ldots,p-1\}$,  we also consider the corresponding Dumont--Thomas numeration system $\rep_{u,r} \colon \mathbb{Z} \to \{\ttt{0},\ttt{1}\}D^*$ from~\cref{def: DTNS for Z equiv r mod p}.

\begin{definition}\label{def:E_j}
Let $j\in \{0,\ldots,p-1\}$.
We let $E_j'$ be the set of letters $c\in A$ such that there exist some integer $k\ge 1$ and some sequence $((m_i,a_i))_{i=0,\ldots,k-1}\in (A^*\times A)^k$ that satisfies the following:
\begin{enumerate}[label=(\alph*)]
\item the sequence $((m_i,a_i))_{i=0,\ldots,k-1}$ is $a$- or $b$-admissible;
\item we have $k\equiv j\bmod p$;
\item the letter $c$ appears at the end of the word $m_0$;
\item if $((m_i,a_i))_{i=0,\ldots,k-1}$ is $b$-admissible,  then $\mu^{k-1}(m_{k-1})\cdots \mu^0(m_0)a_0\neq \mu^k(b)$. 
\end{enumerate} 
\end{definition}

Thinking in terms of the tree $\mathcal{T}_{\mu,b|a}$, conditions (a)--(c) simply describe letters that appear in the tree at some level congruent to $j \bmod p$ and have a younger sibling on that level.
In view of~\cref{fig: a-admissible illustration}, condition (d) excludes letters $c$ that only appear in column $-2$ in the tree $\mathcal{T}_{\mu,b|a}$. To deal with these letters, a dedicated condition is required,  as follows.
Recall that the $j$th level of the tree $\mathcal{T}_{\mu,b|a}$ corresponds to the $j$th iteration of $\mu$ on $b|a$ (by convention, the root of $\mathcal{T}_{\mu,b|a}$ is on level $-1$).

\begin{definition}\label{def:C-2}
Let $j\in \{0,\ldots,p-1\}$.
On level $j$ in the tree $\mathcal{T}_{\mu,b|a}$,  we let $c$ be the letter in column $-2$ and $d$ be the letter in column $-1$ (i.e.,  $d$ is immediately to the right of $c$).
We define the set $E_j$ by 
\[
\begin{cases}
E_j=E_j'\cup\{c\}, & \text{ if $c$ and $d$ are siblings and }\len{\mu^{p-j}(d)}>1;\\
E_j=E_j', & \text{ otherwise.}
\end{cases}
\]

Furthermore, if $c$ and $d$ are siblings, $\len{\mu^{p-j}(d)}=1$ and $j\leq r<p$, we define condition~\cref{eq: condition C} to be 
\begin{equation}
\label{eq: condition C}
\len{\mu^{r-j}(c)} = \len{\mu^{r-j}(e)} \text{ for every letter }e\in E_j.
\end{equation}

As we will see below,  all these values are equal if the numeration system is positional.
\end{definition}

Definition~\ref{def:C-2} introduces the technical conditions necessary to deal with letters in column $-2$. The set $E_j$ is the one we will use in \cref{thm: main}. Sometimes this set is augmented with one letter by this definition, sometimes we care only about the length of a single image, and sometimes the letter is completely irrelevant for positionality.

\begin{example}
\label{ex: phi and phi-squared suite suite}
%voir IMG3313
For the substitution $\mu$ from~\cref{ex: phi and phi-squared},  consider its two-sided periodic point with seed $a|a$,  so $p=2$.
In this case, we obtain from \cref{def:E_j} that $E_0'=\{a\}$ and $E_1'=\{c\}$.  
Note that the letters $b,d$ do not have younger siblings in $\mathcal{T}_{\mu,a|a}$.
Now we apply \cref{def:C-2}. For $j=0$, there is no letter in column $-2$ and $E_0=E_0'$. For $j=1$,  we observe that $c$ is on level $j$ and column $-2$.
It also has $d$ as younger sibling and we have $\len{\mu^{p-j}(d)}=\len{\mu(d)}=1$. Thus $E_1=E_1'$ and condition~\cref{eq: condition C} is defined, but is trivially satisfied for $r=1$ as $c$ is already in $E_1$. 

Consider the substitution $\mu\colon a\mapsto bcd,\, d\mapsto ba,\, b\mapsto b^2,\, c\mapsto b$, its two-sided periodic point of period $2$ with seed $a|b$ and the residue $r=0$. If we go only by~\cref{def:E_j}, we will find $E_0'=E_1'=\{b\}$. However, when we apply~\cref{def:C-2}, we get $E_1=E_1'\cup\{c\}$, and using the following theorem we will now correctly find that the system is not positional (which we can also see from the representations of $-5,\, -3$ and $-2$).
\end{example}

We can now state the main result.

\begin{theorem}
\label{thm: main}
Let $\mu \colon A^* \to A^*$ be a substitution and let $u\in\Per_{\Z}(\mu)$ be a two-sided periodic point with growing seed $b|a$ and period $p\ge 1$. 
The Dumont--Thomas complement numeration system associated with $\mu$,  $u$,  and $r$ is positional if and only if  for every $j\in\{0,\ldots,p-1\}$,  the map $c\mapsto \len{\mu^\ell(c)}$ is constant over $E_j$ for every $\ell$ such that $\ell+j\equiv r\bmod p$, and condition~\cref{eq: condition C} is satisfied for the letters where it was defined.

In this case, the sequences $U,V$ of weights of the numeration system are given as follows: for every $\ell \ge 0$,  we define $U_\ell = \len{\mu^\ell(c)}$ for a letter $c\in E_j$ where $j\in\{0,\ldots,p-1\}$ and $\ell+j\equiv r\bmod p$; and $V_\ell=\len{\mu^\ell(b)}$ for every $\ell\in\N$.
\end{theorem}

\begin{remark}
It may be, although rarely, that $E_j$ is empty.
(For instance, consider the case of $\mu\colon a\mapsto b,\, b\mapsto aa$ and the set $E_1$.)
In this case, if condition~\cref{eq: condition C} was defined for some letter $c$, we may give the weight $\len{\mu^{r-j}(c)}$ to position $r-j$. Otherwise, this means that only the digit $0$ appears at positions congruent to $j\bmod p$ in this numeration system, and as such the weight given to these positions is arbitrary.
\end{remark}

The proof of our main result relies on the following technical lemma.
Note that the conditions (f) and (e) in the statement are those of~\cref{thm: unique admin for left infinite periodic len r,thm: unique admin for right infinite periodic len r} respectively.

\begin{lemma}
\label{lem: technical on c in Ej}
Let $j\in\{0,\ldots,p-1\}$ and fix a letter $c$ in $E_j'$.

\begin{itemize}
\item If there exists a $b$-admissible sequence $((m_i,a_i))_{i=0,\ldots,k-1}$ that satisfies the four conditions (a)--(d) of~\cref{def:E_j}, then there exists a $b$-admissible sequence that satisfies those conditions, plus the extra condition 
\[\text{ (e) } k<p \text{ or }\mu^{p-1}(m_{k-1})\cdots \mu^0(m_{k-p})a_{k-p} \neq \mu^p(b).\]

\item If there exists an $a$-admissible sequence $((m_i,a_i))_{i=0,\ldots,k-1}$ that satisfies the
three conditions (a)--(c) in~\cref{def:E_j}, then there exists an $a$-admissible sequence that satisfies those conditions, plus the extra condition
\[ \text{ (f) } k<p \text{ or }m_{k-1}\cdots m_{k-p}\neq\varepsilon.\]
\end{itemize}
\end{lemma}
\begin{proof}
For the first part, if we assume that $k\geq p$ and $\mu^{p-1}(m_{k-1})\cdots \mu^0(m_{k-p})a_{k-p} = \mu^p(b)$,  then the fact that $u$ has period $p$ implies that $a_{k-p}=b$.
The cropped sequence $((m_i,a_i))_{i=0,\ldots,k-p-1}$ is thus  again $b$-admissible.
We also have $k-p\equiv j\bmod p$ and $c$ is the last letter of $m_0$.
We also have $\mu^{k-p-1}(m_{k-p-1})\cdots \mu^0(m_0)a_0\neq \mu^{k-p}(b)$ for otherwise
\begin{align*}
 \mu^k(b)
 &= \mu^{k-p}(\mu^{p}(b)) = \mu^{k-p} (\mu^{p-1}(m_{k-1})\cdots \mu^0(m_{k-p})b) \\
 &= \mu^{k-1} (m_{k-1})\cdots \mu^{k-p}(m_{k-p}) \mu^{k-p}(b) \\
&= \mu^{k-1} (m_{k-1})\cdots \mu^{k-p}(m_{k-p}) \mu^{k-p-1}(m_{k-p-1})\cdots \mu^0(m_0)a_0,
\end{align*}
which is forbidden by the condition (d) of~\cref{def:E_j}.
We have showed that the shorter sequence $((m_i,a_i))_{i=0,\ldots,k-p-1}$ also satisfies the conditions (a)--(d). Iterating this process, we get the existence of the required sequence.

Similarly, if the sequence is $a$-admissible with $k\geq p$ and $m_{k-1}m_{k-2}\cdots m_{k-p}=\varepsilon$, 
then the fact that $u$ has period $p$ implies that $a_{k-p}=a$.
We conclude in the same way: the cropped sequence $((m_i,a_i))_{i=0,\ldots,k-p-1}$ is thus again $a$-admissible,  we also have $k-p\equiv j\bmod p$, and $c$ is the last letter of $m_0$.
\end{proof}

We now prove our main result.

\begin{proof}[Proof of~\cref{thm: main}]
We start by assuming that the Dumont--Thomas complement numeration system associated with $\mu$, $u$ and $r$ is positional and we let $U,V$ be its sequences of weights as in~\cref{sec:prelim}.
Fix some $j\in\{0,\ldots,p-1\}$.
We prove that $c\mapsto\len{\mu^\ell(c)}$ is constant over $E_j$ for all suitable values of $\ell$. 
Indeed, consider some letter $c\in E_j$. 
From~\cref{def:E_j} and~\cref{def:C-2},  we know that there is some $a$- or $b$-admissible sequence $((m_i,a_i))_{i=0,\ldots,k-1}$ with $k\equiv j \bmod p$ and where $c$ is the last letter of $m_0$. 
Additionally,  if the sequence is $b$-admissible, then we may assume that one of the following holds:
\begin{itemize}
\item we have $\mu^{k-1}(m_{k-1})\cdots \mu^0(m_0)a_0\neq \mu^k(b)$ and either $k<p$ or $\mu^{p-1}(m_{k-1})\cdots \mu^0(m_{k-p})a_{k-p} \neq \mu^p(b)$ (if $c\in E_j'$  and using~\cref{lem: technical on c in Ej});
\item we have $k=j$, $\mu^{k-1}(m_{k-1})\cdots \mu^0(m_0)a_0=\mu^k(b)$ and $|\mu^{p-k}(a_0)|>1$ (if $c\in E_j\setminus E_j'$, by~\cref{def:C-2}).
\end{itemize} 
On the other hand,  if the sequence is $a$-admissible, then from~\cref{lem: technical on c in Ej} we may assume that $k<p$ or $m_{k-1}\cdots m_{k-p}\neq\varepsilon$.

Let now $\ell$ be a positive integer such that $\ell+j\equiv r\bmod p$.
We define two new sequences $((m_n',a_n'))_{n=0,\ldots,k-1+\ell}$ and $((m_n'',a_n''))_{n=0,\ldots,k-1+\ell}$ of elements of $A^*\times A$ as follows.
The first sequence is defined by
\[
(m_{n}',a_{n}')
=
\begin{cases}
(m_{n-\ell},a_{n-\ell}), & \text{ if }n\in\{\ell,\ldots,k-1+\ell\};\\
(\varepsilon,\mu(a_{n+1}')_0), & \text{ if }n\in\{0,\ldots,\ell-1\}.
\end{cases}
\]
For the second sequence,  the definitions of $m''$ and $a''$ are the same,  except at index $n=\ell$ for which we set $m_{\ell}''c=m_{\ell}'$ and $a_{\ell}''=c$. 
(Note that $m'_\ell=m_0$ ends with $c$,  so $m_{\ell}''$ is the prefix of $m_0$ without its last letter,  which we may write $m_0 c^{-1}$.) Thinking in terms of the tree, these sequences correspond to the paths labeled by $w(t+1)0^\ell$ and $wt0^\ell$ from our Sketch~\ref{ske:sketch-main}.

Now, notice that the sequence $((m_n',a_n'))_{n=0,\ldots,k-1+\ell}$ satisfies the conditions of~\cref{thm: unique admin for left infinite periodic len r} in the $b$-admissible case and the conditions of~\cref{thm: unique admin for right infinite periodic len r} in the $a$-admissible case. Indeed, the only case where this is not trivial is the $b$-admissible case with $k<p$, in which case the result is obtained by considering
\[
\len{\mu^{p-1}(m_{k-1})\cdots \mu^{p-k}(m_0)} \leq \len{\mu^{p-1}(m_{k-1})\cdots \mu^{p-k}(m_0) \mu^{p-k}(a_0)}-1 \leq \len{\mu^{p-k}(\mu^k(b))}-1,
\]
where one of the two inequalities must be strict.
The same result is true for the sequence $((m_n'',a_n''))_{n=0,\ldots,k-1+\ell}$, with the sole exception of sequences of this form that are $a$-admissible, with $k< p$, $m_{k-1}=\cdots=m_1=\varepsilon$, and $m_0=c$.

From~\cref{def: DTNS for Z equiv r mod p}, we obtain the following. 
On the one hand, if the  sequences are $b$-admissible, we get that $\ttt{1}\cdot\len{m_{k-1}}\cdots\len{m_0}\cdot \ttt{0}^\ell$ is the representation of the integer
\begin{equation}\label{eq: first integer in main}
\len{\mu^{k-1+\ell}(m_{k-1})\cdots\mu^\ell(m_0)\mu^{\ell-1}(\varepsilon)\cdots\mu^0(\varepsilon)}-\len{\mu^{k+\ell}(b)},
\end{equation}
while $\ttt{1}\cdot\len{m_{k-1}}\cdots\len{m_0c^{-1}}\cdot \ttt{0}^\ell$ is the representation of the integer
\begin{equation}\label{eq: second integer in main}
\len{\mu^{k-1+\ell}(m_{k-1})\cdots\mu^\ell(m_0c^{-1})\mu^{\ell-1}(\varepsilon)\cdots\mu^0(\varepsilon)}-\len{\mu^{k+\ell}(b)}.
\end{equation}
Subtracting~\cref{eq: second integer in main} from~\cref{eq: first integer in main},  we get that $U_\ell=\len{\mu^\ell(c)}$ (recall that $U_\ell$ is the weight in position $\ell$ in the numeration system).
On the other hand, if the sequences are $a$-admissible, we get that $\ttt{0}\cdot\len{m_{k-1}}\cdots\len{m_0}\cdot \ttt{0}^\ell$ is the representation of
\begin{equation}\label{eq: third integer in main}
\len{\mu^{k-1+\ell}(m_{k-1})\cdots\mu^\ell(m_0)\mu^{\ell-1}(\varepsilon)\cdots\mu^0(\varepsilon)}, 
\end{equation}
while $\ttt{0}\cdot\len{m_{k-1}}\cdots\len{m_0c^{-1}}\cdot \ttt{0}^\ell$ is the representation of
\begin{equation}\label{eq: fourth integer in main}
\len{\mu^{k-1+\ell}(m_{k-1})\cdots\mu^\ell(m_0c^{-1})\mu^{\ell-1}(\varepsilon)\cdots\mu^0(\varepsilon)}.
\end{equation}
Subtracting again~\cref{eq: fourth integer in main} from~\cref{eq: third integer in main} leads to the same conclusion. 
In the exceptional case that was singled out above, \cref{eq: fourth integer in main} does not hold anymore but~\cref{eq: third integer in main} becomes that $\ttt{0}^{k-1}\ttt{1}\ttt{0}^\ell$ is a representation of $\len{\mu^\ell(c)}$, which leads to the same conclusion once more.
All in all,  for every $c\in E_j$ and every $\ell$ such that $j+\ell\equiv r\bmod p$, we have shown that $\len{\mu^\ell(c)}=U_\ell$,  thus the map $c\mapsto \len{\mu^\ell(c)}$ is constant over $E_j$ as desired. 

To end the proof of this implication,  we now show that condition~\cref{eq: condition C} is satisfied for the required letters, still in the case of a positional numeration system. 
We let $c$ be a letter for which the hypotheses of condition~\cref{eq: condition C} are fulfilled.
As $c$ is on the $j$th level of the tree $\mathcal{T}_{\mu,b|a}$,  there exists some $b$-admissible sequence $((m_n,a_n))_{n=0,\ldots,j-1}$ with $c$ being the last letter of $m_0$ and $\mu^{j-1}(m_{j-1})\cdots\mu^0(m_0)a_0=\mu^{j}(b)$. 
Since $\len{\mu^{p-j}(a_0)}=1$ and $j\leq r<p$, it follows that the sequence
\[
(m_{j-1},a_{j-1}),\ldots,(m_0,a_0),(\varepsilon, \mu(a_0)),\ldots,(\varepsilon,\mu^{r-j}(a_0))
\]
is $b$-admissible and 
\begin{equation}\label{eq: equality with mu r b}
\mu^{r-1}(m_{j-1})\cdots\mu^{r-j}(m_0)\mu^{r-j}(a_0)=\mu^r(b). 
\end{equation}
Therefore, the representation of $-1$ must be $\ttt{1}\cdot \len{m_{j-1}}\cdots\len{m_0}\cdot \ttt{0}^{r-j}$. On the other hand, we have as above that  $\ttt{1}\cdot\len{m_{j-1}}\cdots\len{m_0c^{-1}}\cdot \ttt{0}^{r-j}$ is the representation of 
\begin{equation}\label{eq: this number main}
\len{\mu^{r-1}(m_{j-1})\cdots\mu^{r-j}(m_0c^{-1})\mu^{r-j-1}(\varepsilon)\cdots\mu^0(\varepsilon)}-\len{\mu^{r}(b)}.
\end{equation}
Utilizing Equality~\cref{eq: equality with mu r b} to substitute $\mu^{r}(b)$,  the integer in~\cref{eq: this number main} is $-1-\len{\mu^{r-j}(c)}$. Comparing the two representations we discussed, we obtain that $\len{\mu^{r-j}(c)}=U_{r-j}$. 
This value is also the value of $\len{\mu^{r-j}(e)}$ for every letter $e\in E_j$ as proven above, so condition~\cref{eq: condition C} is satisfied as expected.

We turn to the converse implication. 
We assume that the map $c\mapsto\len{\mu^\ell(c)}$ is constant over $E_j$ for all suitable values of $j$ and $\ell$ and that  condition~\cref{eq: condition C} is satisfied, and we prove the positionality of the numeration system,  i.e.,  we find suitable sequences of weights. 
We let $U_\ell$ be the constant value imposed on $\mu^\ell$ by the discussed conditions and we also define $V_\ell=\len{\mu^\ell(b)}$ for every $\ell\in\N$.

We now consider an integer $n\in\Z$ and we divide the argument into three different cases. 

\textbf{Case 1}.  
Assume that $n$ is non-negative.
Then~\cref{def: DTNS for Z equiv r mod p} implies that
\begin{equation}\label{eq: rep of n with n positive}
\rep_{u,r}(n)=\ttt{0}\cdot \len{m_{k-1}}\cdots\len{m_0}
\end{equation}
for some $a$-admissible sequence $((m_i,a_i))_{i=0,\ldots,k-1}$ satisfying 
\[
\mu^k(a)_{[0,n-1]}=\mu^{k-1}(m_{k-1})\cdots \mu^0(m_0).
\] 
In particular,  we have that
\begin{equation}\label{eq: n as sum with n positive}
n=\sum_{i=0}^{k-1}\len{\mu^{i}(m_i)}.
\end{equation} 
If $c$ is a letter in $m_i$, we may write $m_ia_i=p_ics_i$ for some words $p_i,s_i$ and we note that the existence of the sequence $(m_{k-1},a_{k-1}),\ldots,(p_ic,(s_i)_0)$ means that $c$ is in $E_{k-i\bmod p}$. 
Since $i+k-i\equiv k\equiv r\bmod p$, we obtain that $\len{\mu^i(c)}=U_i$ for any letter $c$ of $m_i$. 
As a result,  \cref{eq: n as sum with n positive} yields that $n=\sum_{i=0}^{k-1}\len{m_i}U_i$. 
Due to~\cref{eq: rep of n with n positive},  we obtain that the numeration system is positional with weights $(U_i)_{i\in\N}$ as expected.

\textbf{Case 2}.  
Assume that $n$ is negative and different from $-1$.
Again~\cref{def: DTNS for Z equiv r mod p} implies that 
\begin{equation}\label{eq: rep of n with n negative and not -1}
\rep_{u,r}(n)=\ttt{1}\cdot \len{m_{k-1}}\cdots\len{m_0}
\end{equation} 
for some $b$-admissible sequence $((m_i,a_i))_{i=0,\ldots,k-1}$ satisfying 
\[
\mu^k(b)_{[-\len{\mu^k(b)},n-1]}=\mu^{k-1}(m_{k-1})\cdots \mu^0(m_0).
\]

A bit of care is required to prove that all letters in $m_i$ are in $E_{k-i\bmod p}$.
Indeed,  if $c$ is such a letter with $m_ia_i=p_ics_i$ for some words $p_i,s_i$,  then we must verify that either we have \begin{equation}\label{eq:main-thm-case-2}
\mu^{k-i-1}(m_{k-1})\cdots\mu^0(p_ic)(s_i)_0\neq \mu^{k-i}(b),
\end{equation}
or we have an equality but $k-i\leq p$ and $\len{\mu^{p-k+i}(s_i)}>1$. 
The only case where equality \eqref{eq:main-thm-case-2} can hold is if $p_ic=m_i$ and $s_i=a_i$. Next, remember that the sequence $((m_i,a_i))_{i=0,\ldots,k-1}$ must satisfy $\mu^{p-1}(m_{k-1})\mu^{p-2}(m_{k-2}) \cdots \mu^0(m_{k-p}) a_{k-p} \neq \mu^p(b)$. Thus, equality \eqref{eq:main-thm-case-2} can only occur if $k-i<p$. But in this case, it must be that $\len{\mu^{p-k+i}(a_i)}>1$, otherwise we would have 
\begin{align*}
&\mu^{k-i-1}(m_{k-1})\cdots\mu^0(m_i)a_i =  \mu^{k-i}(b) \\
\Rightarrow &\mu^{p-k+i}\left( \mu^{k-i-1}(m_{k-1})\cdots\mu^0(m_i)a_i \right) = \mu^{p}(b)\\
\Rightarrow &\mu^{p-1}(m_{k-1})\cdots \mu^{p-k+i}(m_i) \mu^{p-k+i}(a_i)=\mu^{p}(b)\\
\Rightarrow &\mu^{p-1}(m_{k-1})\mu^{p-2}(m_{k-2}) \cdots \mu^0(m_{k-p}) a_{k-p} = \mu^p(b)
\end{align*}
as $\mu^{p-k+i-1}(m_{i-1})\cdots a_{k-p}$ is a prefix of length at least $1$ of $\mu^{p-k+i}(a_i)$, which has length $1$. 
The last equality we obtain is a contradiction,  so $k-i<p$ and $\len{\mu^{p-k+i}(a_i)}>1$.
Thus,  we have proven that $c$ belongs to $E_{k-i\bmod p}$ in all cases. From there, we get that 
\[
n
=-\len{\mu^k(b)}+\sum_{i=0}^{k-1}\len{\mu^i(m_i)}
=-\len{\mu^k(b)}+\sum_{i=0}^{k-1}\len{m_i}U_i,
\]
where the second equality holds because all the letters of $m_i$ are in $E_{k-i\bmod p}$. 
Note that $\left| \mu^k(b)_{[-\len{\mu^k(b)},n-1]}\right| = n+|\mu^k(b)| $ even in the case where $n-1<-\len{\mu^k(b)}$, due to the condition $-\len{\mu^k(b)}\leq n$ (putting the two inequalities together gives $-\len{\mu^k(b)}=n$).
Due to~\cref{eq: rep of n with n negative and not -1}, we get that the numeration system is positional with sequences $U,V$ of weights as defined above.

\textbf{Case 3}.
Only the case where $n=-1$ remains. 
Here, we have a sequence $((m_i,a_i))_{i=0,\ldots,r-1}$ with $\mu^{r-1}(m_{r-1})\cdots\mu^0(m_0)a_0=\mu^r(b)$. 
All the letters of $m_i$ except the last one are guaranteed to be in $E_{k-i\bmod p}$ as in the previous case.
For the last letter of $m_i$, either it is in $E_{k-i\bmod p}$ or condition~\cref{eq: condition C} applies. In any case, we have $\len{\mu^i(c)}=U_i$,  and the rest of the argument is as for Case 2.

In all three cases, we have shown the numeration system to be positional: for every $i\ge 0$,  the weights $U_i$ are defined by the constant value of $c\mapsto \len{\mu^i(c)}$ over $E_{r-i\bmod p}$, while the weights $V_i$ are equal to $\len{\mu^i(b)}$. This concludes the proof of the main result.
\end{proof}

\begin{remark}
Note that replacing the substitution by one of its powers may lead to the loss of positionality of the corresponding numeration system.
For instance,  consider the substitution $\mu \colon a\mapsto ccd,  b \mapsto cd,  c \mapsto ab,  d\mapsto a$ from~\cref{ex: phi and phi-squared} and its square renamed $\nu\colon a \mapsto ab ab a, b\mapsto ab a, c\mapsto ccd cd,  d\mapsto ccd$. 
From before,  we already know that the Dumont--Thomas numeration system associated with $\mu$ and the periodic point of $\mu$ starting with $a$ of period $2$ is positional.
However,  that associated with $\nu$ and its fixed point starting with $a$ is not.
Indeed,  drawing the first two levels of the tree $\mathcal{T}_{\nu,a}$ we obtain that the representation of $5$ is $\ttt{10}$ and that of $8$ is $\ttt{20}$,  making it impossible to find a suitable evaluation map for the numeration system to be positional. This is as expected from our~\cref{thm: main} as $a$ and $b$ are both in $E_0$ but have images of different lengths.
\end{remark}

\subsection{Particular cases}
\label{sec:particularizations}

In this section, we highlight some general cases where the technicalities of~\cref{sec:mainresult} do not occur, leading to results that are more concise and legible. We also discuss possible simplifications of the substitution at play, as well as a parallel to Bertrand numeration systems. 

From now on, we assume that the alphabet $A$ of the substitution $\mu$ contains no extraneous letters, i.e., that all the letters in $A$ are present in $\mu^n(b|a)$ for some $n$,  where $b|a$ designates the seed of the periodic point $u$ of $\mu$. This is not a strong assumption, and can be easily obtained by restricting the alphabet to only the accessible letters. Among the letters of the alphabet, a particular subset will be of great interest to us, corresponding to the sets $E_j$ of the previous section.

\begin{definition}\label{def: internal letter}
Let $\mu$ be a substitution, let $a$ (resp. $b|a$) be the seed of a chosen periodic point, and let $A$ the alphabet of letters accessible from $a$ (resp. $a$ or $b$).

A letter $c$ is \emph{internal} if there exist $d\in A$, $x\in A^*$,  and $y\in A^+$ such that $\mu(d)=xcy$. 

We let $E_{\mu}$ denote the set of internal letters of $\mu$. 
When there is no ambiguity,  we drop the subscript. 
\end{definition}

Thinking in terms of the tree, the internal letters are those that have a younger sibling somewhere in the tree $\mathcal{T}_{\mu,a}$ (resp. $\mathcal{T}_{\mu,b|a}$).

Our first particular case is the case where the domain $\D$ is equal to $\N$ and the substitution $\mu$ has a fixed point. 
In this case,  the period $p$ is equal to $1$ and there is only one sequence of weights, with no particular care given to the most significant digit in a representation.

\casoriginal

\begin{proof}
Note that since $a$ is the growing seed of a fixed point of $\mu$,  we must have $\mu(a)=ay$ for some non-empty word $y$, so $a\in E_\mu$. 
Note also that since the domain is $\N$, \cref{def:C-2} does not apply. 
It then suffices to apply the proof of~\cref{thm: main} restricted to the domain $\N$,  taking into account that,  since $p=1$ and $r=0$ in our case,  the set $E_\mu$ is equal to the set $E_0$ of~\cref{thm: main}, which is the only set for which we have  conditions to check.
\end{proof}

Another special case is that of primitive substitutions,  even on the complete domain $\D=\Z$. 
Recall that the definition of primitive substitutions was given in~\cref{sec:prelim}, where we mentioned that if $\mu$ is primitive then there exists $k$ such that $a\in\mu^\ell(b)$ for all $a,b\in A$ and $\ell\geq k$.

\casprimitif

\begin{proof}
Let $k$ be an integer such that,  for all letters $c,d \in A$, for all $\ell>k$,  $c$ appears in $\mu^\ell(d)$. 
We then know that every letter appears at every sufficiently large level in the tree $\mathcal{T}_{\mu,a}$. 
In particular, $b$ appears in $\mu^k(a)$, so $\mathcal{T}_{\mu,b}$ appears as a subtree of $\mathcal{T}_{\mu,a}$. Therefore, any letter that appears with a younger sibling on column $-2$ also appears with a younger sibling on some column with non-negative index.
As a result, \cref{def:C-2} can be skipped when determining the sets $E_j$ and condition~\cref{eq: condition C} is never relevantly added to any letter.

Next,  we show that,  for every $j \in \{0,\ldots,p-1\}$,  $E_j=E_\mu$.
Fix some $j \in \{0,\ldots,p-1\}$.
As the inclusion $E_j \subseteq E_\mu$ is clear by~\cref{def:E_j}  we show the other one.
If $c \in E_\mu$,  there exist $d\in A$, $x\in A^*$,  and $y\in A^+$ such that $\mu(d)=xcy$. 
Since $\mu$ is primitive,  $d$ appears in $\mu^{j-1+\ell p}(a)$ for some sufficiently large $\ell$,  so $c$ appears with a younger sibling at some level $j+\ell p$ in the tree $\mathcal{T}_{\mu,b|a}$.
This implies that $c \in E_j$,  as desired.

Finally,  notice that since $a$ is a growing letter of $\mu$, at least one letter of $A$ must have an image by $\mu$ that is of length at least $2$,  thus $E_\mu$ is nonempty.

Now that we have shown that $E_j=E_\mu\neq\emptyset$ for every $j\in\{0,\ldots,p-1\}$ and that condition~\cref{eq: condition C} is not relevant to any letter, applying~\cref{thm: main} leads to the statement.
\end{proof}

The substitutions that generate the positional Dumont--Thomas numeration systems discussed in the two corollaries above turn out to have a special form as we will see. 
This form lends itself well to simplifications and will allow us to link these positional numeration systems to some known families of numeration systems. 
The following lemma shows that these substitutions are equivalent (in the sense that they generate the same Dumont--Thomas numeration system) to a substitution that contains only one internal letter.

\begin{lemma}
\label{lem:simplification-fabrelike}
Let $\mu \colon A^* \to A^*$ be a substitution such that the map $c\mapsto \len{\mu^\ell (c)}$ is constant over $E_{\mu}$ for every integer $\ell\ge 0$. 
Then there exist an alphabet $B\subseteq A$,  a substitution $\nu \colon B^* \to B^*$ such that $\len{E_{\nu}}=1$ and some letters $a',b'\in B$ such that the trees $\mathcal{T}_{\mu,b|a}$ and $\mathcal{T}_{\nu,b'|a'}$ differ only by their labeling.
\end{lemma}

\begin{proof}
Before showing that the statement holds,  we start with some notation.
For an integer $k\ge 0$ and a letter $a\in A$,  we let $\mathcal{T}_{\mu,a}^{\leq k}$ denote the first $k$ levels of the tree $\mathcal{T}_{\mu,a}$ (agreeing that the root is on level $0$). 
Additionally, we extend our definition of trees associated with substitutions to forests associated with substitutions: if $w=w_1\ldots w_n$ is a word over $A$,  then $\mathcal{T}_{\mu,w}$ is the ordered forest where the first tree is $\mathcal{T}_{\mu,w_1}$, the second tree is $\mathcal{T}_{\mu,w_2}$, and so on. 
We define similarly $\mathcal{T}_{\mu,w}^{\leq k}$ for any integer $k\ge 0$.

Now we prove the following claim: for every integer $k\ge 0$ and for every pair of letters $c,d \in E_{\mu}$, the trees $\mathcal{T}_{\mu,c}^{\leq k}$ and $\mathcal{T}_{\mu,d}^{\leq k}$ differ only by their labeling. Proving this for every $k$ will also show it for the entire trees.
We proceed by induction on $k\ge 0$.
The case $k=0$ is trivial as we start with letters.
The case $k=1$ is obtained directly from the hypothesis for $\ell=1$.
Indeed,  $c$ and $d$ must have the same number of children as they are both in $E_\mu$ by assumption.
Now, assume that the claim holds for $k$ and let us prove it for $k+1$. 
We inductively define $\mu(c)=x_1c_1$ and $\mu(c_i)=x_{i+1}c_{i+1}$ for every $i\ge 1$,  where $x_i\in A^*$ and $c_i\in A$ for every $i\ge 1$ (see~\cref{fig:simplification-fabrelike-tc=td}).
Similarly,  we define $\mu(d)=y_1d_1$ and $\mu(d_i)=y_{i+1}d_{i+1}$ for every $i\ge 1$ with the same constraints on $y_i$ and $d_i$.
The induction hypothesis ensures that $\mathcal{T}_{\mu,c}^{\leq k}$ and $\mathcal{T}_{\mu,d}^{\leq k}$ only differ by their labeling. 
Still from the induction hypothesis,  the same is true for $\mathcal{T}_{\mu,x_1}^{\leq k}$ and $\mathcal{T}_{\mu,y_1}^{\leq k}$ and, more generally,  for $\mathcal{T}_{\mu,x_i}^{\leq k+1-i}$ and $\mathcal{T}_{\mu,y_i}^{\leq k+1-i}$ with $i \in \{1, \ldots, k\}$ (letters of the words $x_i,y_i$ are internal by definition). See~\cref{fig:simplification-fabrelike-tc=td} for a more visual explanation.

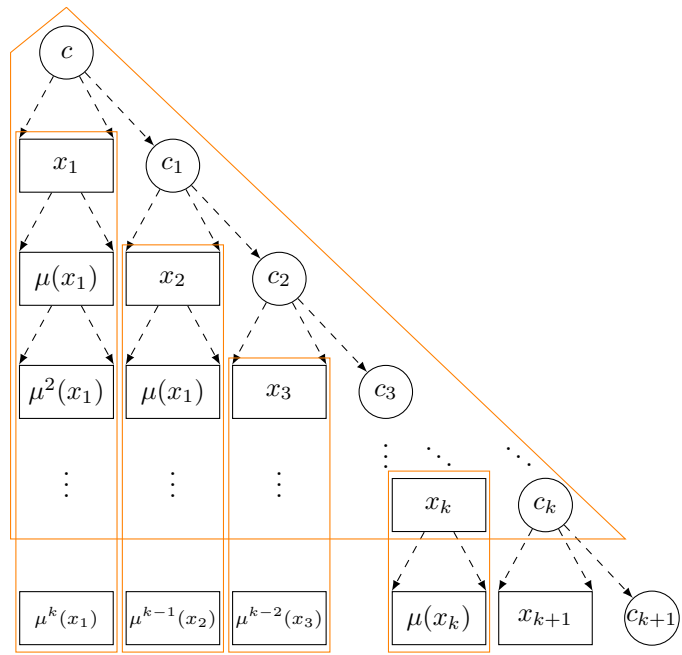
\begin{figure}
\begin{center}
\begin{tikzpicture}[xscale=.70, yscale=1.5, >=latex,
nodeS/.style={rectangle,draw,minimum width=35pt, minimum height=20pt},
nodeC/.style={circle,draw, minimum size=20pt}]
    % nodes positive
    %\node                 (O)   at (0,.5) {};
    \node[nodeC] (c) at (0,0)  {$c$};
    
    \node[nodeS] (x1) at (0,-1) {$x_1$};
    \draw[->,dashed] (c) -- (x1.north west);
    \draw[->,dashed] (c) -- (x1.north east);
    
    \node[nodeC] (c1) at (2,-1)  {$c_1$};
    \draw[->,dashed] (c) -- (c1);

    \node[nodeS] (mx1) at (0,-2) {$\mu(x_1)$};
    \draw[->,dashed] (x1) -- (mx1.north west);
    \draw[->,dashed] (x1) -- (mx1.north east);
    
    \node[nodeS] (x2) at (2,-2) {$x_2$};
    \draw[->,dashed] (c1) -- (x2.north west);
    \draw[->,dashed] (c1) -- (x2.north east);
    
    \node[nodeC] (c2) at (4,-2)  {$c_2$};
    \draw[->,dashed] (c1) -- (c2);

    \node[nodeS] (m2x1) at (0,-3) {$\mu^2(x_1)$};
    \draw[->,dashed] (mx1) -- (m2x1.north west);
    \draw[->,dashed] (mx1) -- (m2x1.north east);
    
    \node[nodeS] (mx2) at (2,-3) {$\mu(x_1)$};
    \draw[->,dashed] (x2) -- (mx2.north west);
    \draw[->,dashed] (x2) -- (mx2.north east);
    
    \node[nodeS] (x3) at (4,-3) {$x_3$};
    \draw[->,dashed] (c2) -- (x3.north west);
    \draw[->,dashed] (c2) -- (x3.north east);
    
    \node[nodeC] (c3) at (6,-3)  {$c_3$};
    \draw[->,dashed] (c2) -- (c3);

    \node at (0,-3.75) {$\vdots$};
    \node at (2,-3.75) {$\vdots$};
    \node at (4,-3.75) {$\vdots$};
    \node at (6,-3.5) {$\vdots$};
    \node at (7,-3.5) {$\ddots$};
    \node at (8.5,-3.5) {$\ddots$};

    \node[nodeS] (xk) at (7,-4) {$x_k$};
    \node[nodeC] (ck) at (9,-4)  {$c_k$};

    \node[nodeS,inner sep=0] (mkx1) at (0,-5) {\scriptsize $\mu^{k}(x_1)$};
    \node[nodeS,inner sep=0] (mkm1x2) at (2,-5) {\scriptsize $\mu^{k-1}(x_2)$};
    \node[nodeS,inner sep=0] (mkm2x3) at (4,-5) {\scriptsize $\mu^{k-2}(x_3)$};
    
    \node[nodeS] (mxk) at (7,-5) {$\mu(x_k)$};
    \draw[->,dashed] (xk) -- (mxk.north west);
    \draw[->,dashed] (xk) -- (mxk.north east);
    
    \node[nodeS] (xkp1) at (9,-5) {$x_{k+1}$};
    \draw[->,dashed] (ck) -- (xkp1.north west);
    \draw[->,dashed] (ck) -- (xkp1.north east);
    
    \node[nodeC,inner sep=0] (ckp1) at (11,-5)  {$c_{k+1}$};
    \draw[->,dashed] (ck) -- (ckp1);

    \draw[orange] (0,0.4)--(10.5,-4.3)--(-1.05,-4.3)--(-1.05,0)--(0,0.4);
    \draw[orange] (-0.95,-5.3) rectangle (0.95,-0.7);
    \draw[orange] (1.05,-5.3) rectangle (2.95,-1.7);
    \draw[orange] (3.05,-5.3) rectangle (4.95,-2.7);
    \draw[orange] (6.05,-5.3) rectangle (7.95,-3.7);

    %\draw[->,dotted] (c) -- node[fill=white,inner sep=2pt] {\scriptsize $0$} (x1);
\end{tikzpicture}
\end{center}
\caption{In the proof of~\cref{lem:simplification-fabrelike},  given a letter $c$,  we inductively define $\mu(c)=x_1c_1$ and $\mu(c_i)=x_{i+1}c_{i+1}$ for every $i\ge 1$,  where $x_i$ is a word and $c_i$ is a letter for every $i\ge 1$.
We draw the tree $\mathcal{T}_{\mu,c}^{\leq k+1}$ as defined in the proof of~\cref{lem:simplification-fabrelike}.
Black rectangles correspond to words (note that we do not give information on the length here) and black circles to letters.  The induction hypothesis tells us that,  up to labeling,  every color-framed subtree is equal to its corresponding subtree in $\mathcal{T}_{\mu,d}^{\leq k+1}$.}
\label{fig:simplification-fabrelike-tc=td}
\end{figure}

As a result,  to conclude the proof of the claim,  we only need to check that $c_k$ and $d_k$ have the same number of children,  i.e.,  $\len{\mu(c_k)}=\len{\mu(d_k)}$.
Since
\[
\len{\mu^{k+1}(c)} = \sum_{i=1}^{k} \len{\mu^{k+1-i}(x_i)} + \len{\mu(c_k)} \quad \text{ and } \quad \len{\mu^{k+1}(d)} = \sum_{i=1}^{k} \len{\mu^{k+1-i}(y_i)} + \len{\mu(d_k)},
\]
we have
\[
\len{\mu(c_k)} 
= \len{\mu^{k+1}(c)} - \sum_{i=1}^{k} \len{\mu^{k+1-i}(x_i)}
= \len{\mu^{k+1}(d)} - \sum_{i=1}^{k} \len{\mu^{k+1-i}(y_i)} 
= \len{\mu(d_k)},
\]
where the second equality results from our hypothesis (all the letters in $x_i$ and $y_i$ belong to $E_\mu$, as do $c$ and $d$). 
Thus the claim is proven.

Now that we have proven that $\mathcal{T}_{\mu,c}$ and $\mathcal{T}_{\mu,d}$ only differ by their labeling, we may replace every instance of the letter $d$ in an image of $\mu$ by the letter $c$ and this will only change the labeling of the tree $\mathcal{T}_{\mu,b|a}$, not its shape. 
If we proceed like this for every letter $d$ in $E_\mu\setminus\{c\}$ and then remove any such superfluous letter from the alphabet $A$,  we will have constructed the desired substitution $\nu$.
\end{proof}

\begin{example}
Consider the primitive substitution $\mu\colon a\mapsto ab,  b\mapsto ba$ (often referred to as the Thue--Morse substitution). 
We note that both $a,b$ are internal letters and $\len{\mu^\ell(c)}=2^\ell$ for $c\in\{a,b\}$ and for every $\ell\ge 0$.
Applying our result gives the substitution $\nu\colon a \mapsto a^2$. The corresponding Dumont--Thomas numeration system (over $\N$) is the usual binary system.

More generally, the Dumont--Thomas numeration system associated with a constant length substitution of length $q\geq 2$ is the $q$-ary numeration system.
\end{example}

As a consequence,  when considering positional Dumont--Thomas numeration systems associated with either primitive substitutions or fixed points on the domain $\N$ (namely,  the assumptions of~\cref{cor: fixed point case,cor: primitive case}) we may look only at substitutions having one single internal letter $e$ by~\cref{lem:simplification-fabrelike}. 
In such a substitution, the image of a letter is determined by the number of repetitions of the internal letter $e$ and the choice of final letter. 

If we let $a$ be the seed of the periodic point $u$, it must be that $\mu^\ell(a)$ starts with $e$ for some $\ell$. In order for $u$ to be a periodic point, we must have either $e=a$ (Case 1) or a chain of letters that all have images of length $1$,  say $e\mapsto b_1\mapsto b_2 \mapsto \cdots \mapsto b_k\mapsto a$ (Case 2).  
In all other cases, no iterated image of $e$ starts with $a$, and $u$ cannot be a periodic point as a result.

In what follows, we will try to link these substitutions to established literature on numeration systems. However, note that in most common systems, representations may take any length, which corresponds in our case to the substitution having a fixed point, and therefore to Case 1 above.
As a result, we will solely focus on this case where $e=a$ is both the seed of the periodic point and the only internal letter. In this case, with the extra condition to operate on the minimal alphabet, the substitution is equivalent to one of the form 
\begin{equation}
\label{eq: Fabre like substitution}
\mu : a_1 \mapsto a_1^{d_1}a_2, a_2\mapsto a_1^{d_2} a_3, \ldots,  a_n \mapsto a_1^{d_n}a_k,
\end{equation}
where $n$ is a non-negative integer, $\{a_1,\ldots, a_n\}$ is the alphabet of the substitution, $k\in\{1,\ldots,n\}$, $d_1>0$ and $d_i$ is a non-negative integer for every $i\in\{1,\ldots, n\}$. 
An example of such a substitution is the Fibonacci (resp.,  Tribonacci) substitution given by $\varphi\colon a\mapsto ab,  b\mapsto a$ (resp.,  $\tau \colon a \mapsto ab,  b \mapsto ac,  c \mapsto a$),  for which $n=2$, $a_1=a$,  $a_2=b$,  $d_1=1$,  $d_2=0$, and $k=1$ (resp.,  $n=3$,  $a_1=a$,  $a_2=b$,  $a_3=c$, $d_1=d_2=1$,  $d_3=0$,  and $k=1$).
See also their generalization to \emph{generic $n$-bonacci} substitutions in~\cite[Ex 2.11]{Rigo-2014-2}.

The similarity with the substitutions studied by Fabre in~\cite{Fabre-1995} is striking,  so we will call \emph{Fabre-like} the substitutions of the form~\cref{eq: Fabre like substitution}. We will now study this particular class of substitutions in additional details.
First, we note the following property.

\begin{proposition}
\label{pro: unique decomposition}
Let $\mu$ and $\mu'$ be Fabre-like substitutions and consider their respective fixed points $u$ and $u'$ with respective seeds $a$ and $a'$. 
If the Dumont--Thomas numeration systems over $\N$ corresponding to $u$ and $u'$ respectively and $r=0$ have the same sequence of weights,  then every natural number has the same representation in both numeration systems.
\end{proposition}

In a sense, the weight sequence completely determines the numeration system provided that it is based on a Fabre-like substitution. In what follows, we will call two numeration systems over $\N$ \emph{equal} if every natural number is represented by the same word in both systems. Note that this property does not hold in general as shown in the next example. 

\begin{example}
Consider the substitutions $\mu \colon a\mapsto bcd,\, b\mapsto ef,\, c\mapsto e^2 g,\, d\mapsto d,\, e\mapsto ad,\, f\mapsto a^2d,\, g\mapsto a$ and $\mu'$ defined in the same way but with $\mu'(a)=cbd$ instead of $bcd$. 
Then, if $u$ and $u'$ are the periodic points of period $p=3$ starting with $a$ of $\mu$ and $\mu'$ respectively and $r=0$, the corresponding Dumont--Thomas numeration systems associated with $u$ and $r$ and with $u'$ and $r$ have the same sequence of weights,  but the representations of some numbers differ: for example,  we have $\rep_{u,0}(4)=\texttt{012}$,  $\rep_{u',0}(4)=\texttt{020}$,  $\rep_{u,0}(9)=\texttt{120}$,  and $\rep_{u',0}(9)=\texttt{112}$.
\end{example}

\begin{proof}[Sketch of the proof of~\cref{pro: unique decomposition}]
From the equality of the weight sequences, we deduce that $\len{\mu^\ell(a)}=\len{\mu'^{\ell}(a')}$ for every integer $\ell$ as $a\in E_\mu$ and $a'\in E_{\mu'}$. Then, we prove that, for every $k\ge 0$,  the pair of trees $\mathcal{T}_{\mu,a}^{\leq k}$ and $\mathcal{T}_{\mu',a'}^{\leq k}$ differ only by their labeling, using the same method as in the proof of~\cref{lem:simplification-fabrelike}. Since the two trees differ only by their labeling,  $\rep_{u,0}(n)=\rep_{u',0}(n)$ for every  natural number $n$.
\end{proof}

Let us quickly recall some other milestones of numeration systems. 
For the reader not familiar with the classical theory and definitions,  see,  for instance,  \cite[Chapter~2]{CANT10} and~\cite[Chapter~2]{Rigo-2014-2}.

In $1957$,  R\'{e}nyi~\cite{Renyi-1957} introduced $\beta$-numeration systems, a way to represent positive real numbers with the help of a real base $\beta>1$. 
A \emph{$\beta$-representation} of a non-negative real number $x$ is an infinite sequence $(a_i)_{i\ge 1}$ over $\N$ such that $x=\sum_{i=1}^\infty \frac{a_i}{\beta^{i}}$. The most commonly used algorithm in order to obtain such a sequence of digits $(a_i)_{i\ge 1}$ is the \emph{greedy algorithm}. 
The corresponding distinguished $\beta$-representation of a given $x\in[0,1]$ is called the $\beta$-\emph{expansion} of $x$ and is obtained as follows: set $r_0=x$ and for all $i\ge 1$, let $a_i=\floor{\beta\, r_{i-1}}$ and $r_i=\beta\, r_{i-1} - a_i$. The $\beta$-expansion of $x$ is the infinite word $d_{\beta}(x)=(a_i)_{i\ge 1}$ written over the alphabet $\{0,\ldots, \floor{\beta}\}$. 
The $\beta$-expansion $d_{\beta}(1)$ of $1$ as well as its \emph{quasi-greedy $\beta$-expansion} $d_\beta^*(1)=\lim_{x\to 1^-}d_\beta(x)$ play crucial roles.
Note that $d_{\beta}^*(1)$ is equal to $(d_1\cdots d_{\ell-1}(d_{\ell}-1))^\omega$ if $d_{\beta}(1)=d_1\cdots d_{\ell}0^\omega$,  and $d_{\beta}(1)$ otherwise.
Of particular interest are the so-called \emph{Parry numbers}, which are real numbers $\beta$ such that $d_{\beta}(1)$ is either finite ($\beta$ is then a \emph{simple Parry} number) or ultimately periodic~\cite{Parry-1960}. 

Numeration systems are also used to represent non-negative integers. 
Given an increasing integer sequence $U=(U_i)_{i\ge 0}$ such that $U_0=1$ and the quotients $\frac{U_{i+1}}{U_i}$ are bounded,  we let $\rep_U(n)$ denote the \emph{(greedy) $U$-representation} of $n\in\N$ and defined by the unique word $a_1 \cdots a_\ell$ over $\N$ such that $n=\sum_{i=1}^\ell a_i U_{\ell-i}$, $a_1 \neq 0$ and for all $j \in \{1,\ldots,\ell\}$, $\sum_{i=j}^\ell a_iU_{\ell-i}< U_{\ell-j+1}$. 
Numeration systems of this kind are called \emph{positional greedy numeration systems}, as they are positional and rely on a greedy algorithm. For brevity, we will also call them \emph{$U$-systems}.
In this framework,  among all $U$-representations of a given length $i$ (for $i\ge 1$),  that of $U_i-1$ is maximal with respect to the lexicographic order on words. It is also worth noting the language of these systems is suffix-closed, that is, a suffix of a representation is itself a representation (when allowing for leading zeroes). In particular, the suffix of length $i$ of some representation is always lexicographically less than or equal to the word $\rep_U(U_i-1)$.

In an article of Bertrand-Mathis~\cite{Bertrand-Mathis-1989}, later corrected by Charlier, Cisternino and Stipulanti~\cite{Charlier-Cisternino-Stipulanti-2022}, a link is made between R\'{e}nyi numeration systems and some greedy positional numeration systems for natural numbers. 
If we write $d_{\beta}^*(1)=d_1d_2\cdots $,  we define the sequence of weights $(U_i)_{i \ge 0}$ by $U_0=1$ and the relation $U_i = d_1U_{i-1}+d_2U_{i-2}+\ldots+d_i U_0+1$ for each $i \ge 1$.
In the case where $\beta$ is a Parry number, this relation can be simplified to a linear recurrence relation for $(U_i)_{i \ge 0}$. 
This weight sequence, when combined with a greedy algorithm, defines a positional numeration system called a \emph{canonical Bertrand numeration system},  which enjoys properties such as the numeration language having the same set of factors as the language of the R\'{e}nyi numeration system with base $\beta$. 
Two other notable properties of such numeration systems are the following:

\begin{property}
\label{property: extending with 0} 
A word $w$ is the $U$-representation of some natural number if and only if $w0$ is.
\end{property}

\begin{property}
\label{property: lex greatest rep} 
The lexicographically greatest $U$-representations of every length $\{\rep_U(U_i-1) \mid i\ge 1\}$ are all prefixes of one another.
\end{property}

However, these properties characterize Bertrand numeration systems only if we consider \emph{all} Bertrand numeration systems: in addition to canonical ones,  we consider the \emph{non-canonical} Bertrand numeration systems,  obtained just like the canonical ones but with the word $d_{\beta}^*(1)$ replaced by $d_{\beta}(1)$,  and the \emph{trivial} Bertrand numeration system based on the weights $U_{i}=i+1$,  with $i \ge 0$. 
This is the nature of the correction brought by Charlier et al. in~\cite{Charlier-Cisternino-Stipulanti-2022}.

A few years after Bertrand-Mathis' paper, Fabre~\cite{Fabre-1995} introduced another way to view canonical Bertrand numeration systems. If $\beta$ is a Parry number with $d_\beta^*(1)= d_1\cdots d_n(d_{n+1}\cdots d_{n+m})^\omega$ for some $m,n$,  we introduce the substitution
\begin{align}
\label{eq: substitution beta}
\mu_{\beta} \colon 1\mapsto 1^{d_1}2,\, 2\mapsto 1^{d_2}3,\ldots, (n+m-1)\mapsto 1^{d_{n+m-1}}(n+m),  (n+m)\mapsto 1^{d_{n+m}}(n+1).
\end{align}

Fabre then shows that $|\mu_\beta^{\ell}(1)|$  is the integer $U_{\ell}$ defined for the canonical Bertrand numeration system, and his~\cite[Theorem~2]{Fabre-1995} establishes the equality between the Dumont--Thomas numeration system associated with $\mu_{\beta}$ and the canonical Bertrand numeration system associated with $\beta$.

Note that Fabre defines his substitutions not by using the word $d_{\beta}^*(1)$ like we do but using the word $d_{\beta}(1)$. Still, the substitution in~\cref{eq: substitution beta} is the same as the one defined by Fabre. 
The case where $d_{\beta}(1)$ is finite corresponds to the case where $n=0$ with our notation. 

Another way of defining a substitution from the expansion of $1$ is to consider directly $d_{\beta}(1)$. When $d_{\beta}(1)$ is finite, this leads to a new substitution as we may consider $d_{\beta}(1)=d_1\cdots (d_m+1) 0^\omega$ rather than $d_{\beta}(1)^*=(d_1\cdots d_m)^\omega$. This leads to defining the substitution 
\begin{align}
\label{eq: substitution beta prime}
\mu_{\beta}' \colon 1\mapsto 1^{d_1}2,\,\ldots, (m-1)\mapsto 1^{d_{m-1}}m,  m\mapsto 1^{d_m+1}(m+1), (m+1)\mapsto (m+1).
\end{align}

As an aside, we note that the term ``conjugate substitutions'' used by Fabre in~\cite[Section~3.1]{Fabre-1995} can be thought of in terms of the trees associated with those substitutions. 
Following Fabre's notation,  for two substitutions $\mu \colon A^* \to A^*$ and $\nu\colon B^* \to B^*$ with seeds $a\in A$ and $b\in B$ respectively,  we write $\mu\rightarrow\nu$ if there exists a morphism $h \colon A^* \to B^*$ such that  $h(a)=b$ and $h(\mu(c))=\nu(h(c))$ for every $c\in A$.
We then say that $\mu$ and $\nu$ are \emph{conjugates} if there exists a third substitution $\tau$ such that $\tau\rightarrow\mu$ and $\tau\rightarrow\nu$.

\begin{proposition}
\label{pro: conjugate substitutions}
Two substitutions $\mu \colon A^* \to A^*$ and $\nu\colon B^* \to B^*$ with respective seeds $a$ and $b$ are conjugate exactly when the trees $\mathcal{T}_{\mu,a}$ and $\mathcal{T}_{\nu,b}$ are equal up to relabeling. 
\end{proposition}
\begin{proof}
On the one hand, if the trees $\mathcal{T}_{\mu,a}$ and $\mathcal{T}_{\nu,b}$ are equal up to relabeling, we may construct a tree with the same graph structure but with labels in $A\times B$ obtained by simply joining the labels of the corresponding nodes in the two trees. 
Then, this new tree gives a substitution $\mu\times\nu$ on $A\times B$ that is such that $\mu\times\nu \rightarrow \mu$ and $\mu\times\nu \rightarrow \nu$,  which implies that $\mu$ and $\nu$ are conjugate. 

On the other hand, if $\tau\rightarrow\mu$ for some substitution $\tau$,  applying the morphism $h$ on the label of every node in the tree associated with $\tau$ gives the tree associated with $\mu$ (which must have the same shape since we only have changed the labels). 
By applying this result twice,  we conclude that the trees associated with two conjugate substitutions have the same shape and differ only by their labeling.
\end{proof}

With these milestones recalled, let us go back to the study of our Dumont--Thomas numeration systems based on Fabre-like substitutions. We first note that all Bertrand numeration systems associated with Parry numbers are a particular case of Dumont--Thomas numeration systems (and not just the canonical ones as proven by Fabre).

\begin{proposition}\label{prop: Bertrand-implique-DTFabrelike}
Every Bertrand numeration system associated with a Parry number is equal to some Dumont--Thomas numeration system associated with a Fabre-like substitution.
\end{proposition}

\begin{proof}
It is clear that the substitution $\mu_\beta$ in~\cref{eq: substitution beta} introduced by Fabre is Fabre-like when $\beta$ is a Parry number (note that the first digit of $d_{\beta}^*(1)$ is nonzero, so $d_1>1$ as expected). Therefore, the canonical Bertrand numeration systems can all be seen as special cases of our Dumont--Thomas numeration systems. 

The non-canonical Bertrand numeration systems introduced in~\cite{Charlier-Cisternino-Stipulanti-2022} also fall in this framework. 
The substitution $\mu_{\beta}'$ defined above \eqref{eq: substitution beta prime} generates a Dumont--Thomas numeration system equal to the non-canonical Bertrand numeration system when $\beta$ is a simple Parry number. 

Since the trivial Bertrand numeration system based on the weight sequence $(i+1)_{i\ge 0}$ is the Dumont--Thomas numeration system associated with the substitution $\mu\colon a\mapsto ab,\, b\mapsto b$,  which is Fabre-like,  we conclude the proof of the statement.
\end{proof}

Although our Dumont--Thomas numeration systems clearly verify the two Properties~\ref{property: extending with 0} and~\ref{property: lex greatest rep} 
that characterize Bertrand numeration systems, the converse of~\cref{prop: Bertrand-implique-DTFabrelike} is not true, as we see in the following example.

\begin{example}
\label{ex: DMNS not Bertrand}
Consider the  Fabre-like substitution $\mu\colon a\mapsto aab,\, b\mapsto aaaa$ with fixed point $u=aabaabaaaa\cdots$. This substitution is Fabre-like, of the form \eqref{eq: Fabre like substitution}.
The corresponding Dumont--Thomas numeration system is positional, with the sequence of weights starting by $1,3,10,32,\ldots$ (obtained with \cref{cor: fixed point case}). 
This numeration system cannot be Bertrand. If it were, the representation of the number $9$ given these weights would be $\ttt{30}$ since all Bertrand numeration systems follow the greedy algorithm described above. However, the representation of $9$ in our system is $\rep_{u,0}(9)=\ttt{23}$.

Note also that the dominant eigenvalue $\beta=1+\sqrt{5}$ of the substitution $\mu$ is not a Parry number. As such, the Bertrand numeration system associated with $\beta$ cannot be equal to any Dumont--Thomas numeration system associated with a substitution on a finite alphabet.
\end{example}

The catch, of course, is that Dumont--Thomas numeration systems, while positional,  do not represent numbers by applying a greedy algorithm on the sequence of weights,  but rather by applying a greedy algorithm on the factorization of their fixed point, which can lead to different results. 
This can be understood with an adaptation of the \emph{Parry condition}. To recall, this condition (first seen in \cite[Corollary 1]{Parry-1960}) states that a word $d_1d_2\cdots$ is equal to $d_{\beta}(1)$ for some $\beta>1$ if, and only if, $d_1d_2\cdots \succ 10^\omega$ and $d_1d_2\cdots \succ d_i d_{i+1}\cdots$ for every $i\geq 1$.
In the case of~\cref{ex: DMNS not Bertrand}, if the substitution $\mu$ were the Fabre substitution associated with some Parry number $\beta$, we would have $d_{\beta}^*(1)=(23)^\omega$ and $d_{\beta}(1)=240^\omega$, which contradicts the Parry condition. Thus the system cannot be equal to a Bertrand numeration system.

In fact, the Parry condition (adapted for use with $d_{\beta}^*(1)$ instead of $d_{\beta}(1)$) is all that is necessary to guarantee that the system is greedy and therefore equal to a Bertrand numeration system, as we will see in the following result.

\begin{proposition}
\label{pro: equivalence with Parry condition}
Let $\mu$ be a Fabre-like substitution as in~\cref{eq: Fabre like substitution}.
Construct the word $d_1d_2\cdots = d_1\cdots d_{k-1}(d_k\cdots d_n)^\omega$.
The Dumont--Thomas numeration system associated with $\mu$ and the seed $a_1$ is equal to a Bertrand numeration system if and only if we have $d_id_{i+1}\cdots \preccurlyeq_{\text{lex}} d_1d_2\cdots$ for each $i\ge 1$.
\end{proposition}

\begin{proof}
Assume first that $d_id_{i+1}\cdots \succ_{\text{lex}} d_1d_2\cdots$ for some $i\ge 1$. 
In this case, we must have $d_{i}\cdots d_{j-1} \succ_{\text{lex}} d_1\cdots d_{j-i}$ for some $i,j$. 
However, it is not hard to see that the lexicographically greatest representations of length $j$ and $j-i$ in our numeration system must be $d_1\cdots d_j$ and $d_1\cdots d_{j-i}$ respectively. 
Since the suffix of length $j-i$ of the former word is lexicographically greater than the latter, our system cannot be equal to a greedy positional numeration system (it is not suffix-closed), and therefore cannot be equal to a Bertrand numeration system.

For the other direction, assume that
\begin{equation}\label{eq: parry-condition-db*}
d_id_{i+1}\cdots \preccurlyeq_{\text{lex}} d_1d_2\cdots,  \quad \forall i\ge 1.
\end{equation}
holds. 
We consider two cases depending on the periodicity of $d_1d_2\cdots$.

As a first case, if $d_1d_2\cdots$ is not purely periodic, the inequality sign can be replaced by a strict inequality in~\cref{eq: parry-condition-db*} (which corresponds to the Parry condition). 
If $d_1d_2\cdots = 10^\omega$, then $\mu\colon a\mapsto ab,\, b\mapsto b$ and our Dumont--Thomas system is equal to the trivial Bertrand numeration system associated with weights $(i+1)_{i \ge 0}$. 
If $d_1d_2\cdots$ ends with $0^\omega$ and is not $10^\omega$, then $d_1d_2\cdots$ is equal to $d_{\beta}(1)$ for some simple Parry number $\beta$.
Then,  $\mu$ is equal to the Fabre substitution $\mu_{\beta}'$ modified to fit the non-canonical Bertrand system (as in the proof of~\cref{prop: Bertrand-implique-DTFabrelike}) and our Dumont--Thomas numeration system is equal to the non-canonical Bertrand numeration system associated with $\beta$. Finally,  if $d_1d_2\cdots$ does not end in $0^\omega$, then it is equal to $d_{\beta}(1)=d_{\beta}^*(1)$ for some non-simple Parry number $\beta$. 
In this case, $\mu=\mu_{\beta}$ and our Dumont--Thomas numeration system is equal to the canonical Bertrand numeration system associated with $\beta$.

As a second case, if $d_1d_2\cdots $ is purely periodic with minimal period $\ell$ and is equal to $(d_1\cdots d_\ell)^\omega$, then $d_1\cdots d_{\ell-1}(d_\ell+1)0^\omega$ also satisfies~\cref{eq: parry-condition-db*}, this time with a strict inequality (see~\cite[Lemma~4]{Charlier-Cisternino-Stipulanti-2022}). 
This new word is equal to $d_{\beta}(1)$ for some simple Parry number $\beta$, thus $d_1d_2\cdots$ is equal to $d_{\beta}^*(1)$ for the same $\beta$,  our $\mu$ is the associated Fabre substitution,  and our numeration system is the canonical Bertrand numeration system associated with $\beta$. 

In all cases, we have shown that the Dumont--Thomas numeration system associated with $\mu$ is equal to some Bertrand numeration system, as desired.
\end{proof}

The results of this section can be summarized as follows.

\resumeliens

\begin{proof}
\begin{enumerate}[label=(\alph*)]
\item If the Dumont--Thomas system associated with $\mu$ and a right-infinite fixed point is positional, then by \cref{cor: fixed point case}, the map $ c\mapsto \len{\mu^\ell(c)}$ is constant over the set $E_{\mu}$ of internal letters for every $\ell$. Then, by \cref{lem:simplification-fabrelike}, we can replace $\mu$ by another substitution $\nu$ that has only one internal letter, in such a way that the trees $T_{\mu}$ and $T_{\nu}$ have the same shape. Since the trees have the same shape, the Dumont--Thomas numeration systems are equal. Finally, because $\nu$ has only one internal letter, it must have a fixed point just like $\mu$, and so $\nu$ is the desired Fabre-like substitution.

\item First, if a numeration system is simultaneously a Dumont--Thomas system and a $U$-system, it is a positional Dumont--Thomas system and is therefore associated with a Fabre-like substitution as per the previous item. Then, depending on the parameters $n,k$ and $d_1,\ldots, d_{n}$ (defined in~\cref{eq: Fabre like substitution}) of this Fabre-like substitution, \cref{pro: equivalence with Parry condition} and its proof assert that the system is either not a greedy positional numeration system (which contradicts our hypothesis here) or is a Bertrand numeration system associated with a Parry number.

For the other direction, \cref{prop: Bertrand-implique-DTFabrelike} states that a Bertrand system associated with a Parry number is equal to a Dumont--Thomas system, and all Bertrand systems are $U$-systems by definition.

Finally, Bertrand systems associated with Parry numbers are $U$-systems for a linearly recurrent sequence $U$ per \cite[Proposition 1]{Charlier-Cisternino-Stipulanti-2022}. We also see that the substitution $a\mapsto ab,\, b\mapsto b$ associated with the trivial Bertrand system is not primitive, and neither is the substitution $\mu_{\beta}'$ associated with a non-canonical Bertrand system, as $m+1$ is not strongly connected to the other letters (recall the definition~\eqref{eq: substitution beta prime}). On the other hand, the substitution $\mu_{\beta}$ associated with the canonical Bertrand system is primitive, because the numbers $d_{n+1},\ldots,d_{n+m}$ in \eqref{eq: substitution beta} cannot all be $0$ for otherwise $d_{\beta}^*(1)$ would end in infinitely many zeroes.
\end{enumerate}
\end{proof}

While we have focused on linking Dumont--Thomas numeration systems with the existing Bertrand numeration systems, it can also be noted that positional Dumont--Thomas systems could be used to generalize Bertrand systems.

\begin{example}
\label{ex: two alternate substitutions}
Let us consider the substitution $\mu \colon \{a,b,c,A,B,C,D\}^* \to \{a,b,c,A,B,C,D\}^*$ defined by $a\mapsto AAB$,  $b\mapsto C$,  $c \mapsto D$,  $A \mapsto ab$,  $B \mapsto c$,  $C \mapsto ab$,  $D \mapsto ac$.
%\[
%    \mu\colon\left\{
%    \begin{array}{ll}
%    a \mapsto AAB & A \mapsto ab\\
%    b \mapsto C   & B \mapsto c\\
%   c \mapsto D   & C \mapsto ab\\
%                  & D \mapsto ac
%    \end{array}
%    \right.
%    .
%\]
This substitution is irreducible rather than primitive,  i.e., its adjacency matrix is irreducible but not primitive.
% with an index of imprimitivity of $2$. \todo{Sens ? Def ?} 
It is not Fabre-like and \cref{cor: fixed point case,cor: primitive case} cannot be applied to it. However, the associated Dumont--Thomas numeration systems are still positional,  as there is only one internal letter at each of the two parities.
They do not verify \cref{property: extending with 0,property: lex greatest rep} that characterize Bertrand numeration systems, but they satisfy adapted versions of these properties, as follows.
\end{example}

\begin{property}
\label{property: extending with 0 prime} 
A word $w$ is the $U$-representation of some natural number if and only if $w0^2$ is.
\end{property}

\begin{property}
\label{property: lex greatest rep prime} 
The lexicographically greatest $U$-representations of odd length $\{\rep_U(U_{2i+1}-1) \mid i\ge 0\}$ are all prefixes of one another, as are the lexicographically greatest representations of even length.
\end{property}

These two properties, with $2$ replaced by an arbitrary parameter $p$, can be taken as the definition for an extension of Bertrand numeration systems.

\section{Future directions}

We end this article by presenting several directions of research related to our topic.

First, we ask about the practical aspects of the results presented here. 

\begin{problem}
Can the results in this paper be turned into an algorithmic process and, if so, what would its algorithmic complexity be?
\end{problem}

Second, we note that abstract numeration systems may be defined using trees (by representing a number using the label of a path in a tree) even when those trees are not defined using an underlying substitution. This process corresponds to abstract numeration systems whose language is prefix-closed and right-extendable. 
We thus raise the following question,  which would lead us one step closer to a characterization of positionality for general abstract numeration systems

\begin{problem}
Can our methods be adapted to the case of abstract numeration systems whose language is prefix-closed and right-extendable?
\end{problem}

Third, we have used the term Dumont--Thomas \emph{complement} numeration systems by analogy with the two's complement numeration system.
We naturally wonder the following. 

\begin{problem}
What properties are shared between our Dumont--Thomas complement and two's complement numeration systems? In particular, can addition be performed similarly to the addition of natural numbers?
\end{problem}

For the analogue of the two's complement numeration system in the Fibonacci context,  this question was answered positively in~\cite{Labbe-Lepsova-2023a}. 

Finally,  we mention that the two adapted properties stated in~\cref{ex: two alternate substitutions} could be used as the basis for a generalization of Bertrand numeration systems. We could then obtain properties similar to~\cref{prop: Bertrand-implique-DTFabrelike} in this new framework. 
Research is being conducted in this direction by \'{E}milie Charlier,  Savinien Kreczman,  and Manon Stipulanti.

\section*{Acknowledgments}

We thank \'Emilie Charlier for useful discussions.

We thank the anonymous reviewers for their detailed report.

Savinien Kreczman is supported by the FNRS Research Fellow grant 1.A.789.23F.
Manon Stipulanti is an FNRS Research Associate supported by the Research grant 1.C.104.24F. 

This work was partly funded from France's Agence Nationale de la Recherche
(ANR) project IZES (ANR-22-CE40-0011) and
by the Research Council of Finland under project number  354965.
It was also supported by grants from the
\emph{Symbolic Dynamics and Arithmetic Expansions}
(SymDynAr) Project, co-funded by ANR (ANR-23-CE40-0024)
and FWF (\href{https://dx.doi.org/10.55776/I6750}{I 6750}),
the Austrian Science Fund.

%
% ---- Bibliography ----
%
% BibTeX users should specify bibliography style 'splncs04'.
% References will then be sorted and formatted in the correct style.
%We strongly encourage you to include DOIs (Digital Object Identifiers) in your references.
 \bibliographystyle{plainurl}
\bibliography{bibliography.bib}

@article {Labbe-Lepsova-2023a,
    AUTHOR = {Labb\'e, S\'ebastien and Lep{\v s}ov\'a, Jana},
     TITLE = {A {F}ibonacci analogue of the two's complement numeration
              system},
   JOURNAL = {RAIRO Theor. Inform. Appl. (RAIRO:ITA)},
  FJOURNAL = {RAIRO Theoretical Informatics and Applications (RAIRO: ITA)},
    VOLUME = {57},
      YEAR = {2023},
     PAGES = {Paper No. 12, 23},
      ISSN = {2804-7346},
   MRCLASS = {68Q45 (11A67 11B39)},
  MRNUMBER = {4678565},
       DOI = {10.1051/ita/2023007},
}

@article {MR4836876,
    AUTHOR = {Labb\'e, S\'ebastien and Lep{\v s}ov\'a, Jana},
     TITLE = {Dumont-{T}homas complement numeration systems for {$\Bbb Z$}},
   JOURNAL = {Integers},
  FJOURNAL = {Integers. Electronic Journal of Combinatorial Number Theory},
    VOLUME = {24},
      YEAR = {2024},
     PAGES = {Paper No. A112, 27},
      ISSN = {1553-1732},
   MRCLASS = {11A63 (11B85 37B52 68R15)},
  MRNUMBER = {4836876},
	   DOI = {10.5281/zenodo.14340125}
}

@book {Baake-Grimm-2013,
    AUTHOR = {Baake, Michael and Grimm, Uwe},
     TITLE = {Aperiodic order. {V}ol. 1},
    SERIES = {Encyclopedia of Mathematics and its Applications},
    VOLUME = {149},
 PUBLISHER = {Cambridge University Press, Cambridge},
      YEAR = {2013},
     PAGES = {xvi+531},
       DOI = {10.1017/CBO9781139025256},
       URL = {https://doi.org/10.1017/CBO9781139025256},
}

@incollection {Carton-Couvreur-Delacourt-Ollinger-2024,
    AUTHOR = {Carton, Olivier and Couvreur, Jean-Michel and Delacourt,
              Martin and Ollinger, Nicolas},
     TITLE = {Linear recurrence sequence automata and the addition of
              abstract numeration systems},
 BOOKTITLE = {Combinatorics on words},
    SERIES = {Lecture Notes in Comput. Sci.},
    VOLUME = {15729},
     PAGES = {70--82},
 PUBLISHER = {Springer, Cham},
      YEAR = {2025},
       DOI = {10.1007/978-3-031-97548-6\_7},
}

@article {Dumont-Thomas-1989,
    AUTHOR = {Dumont, Jean-Marie and Thomas, Alain},
     TITLE = {Syst\`emes de num\'eration et fonctions fractales relatifs aux substitutions},
   JOURNAL = {Theoret. Comput. Sci.},
    VOLUME = {65},
      YEAR = {1989},
    NUMBER = {2},
     PAGES = {153--169},
       DOI = {10.1016/0304-3975(89)90041-8},
       URL = {https://doi.org/10.1016/0304-3975(89)90041-8},
}

@phdthesis{Lepsova-2024,
	type = {{PhD} {Thesis}},
	title = {Substitutive structures in combinatorics, number theory, and discrete geometry},
	school = {Universit\'e de Bordeaux and Czech Technical University in Prague},
    AUTHOR = {Lep\v{s}ov\'{a}, Jana},
	YEAR = {2024},
	note = {Available online at \url{https://theses.fr/2024BORD0083} or \url{https://theses.hal.science/tel-04679032}},
}

@article {Lecomte-Rigo-2001,
    AUTHOR = {Lecomte, Pierre B. A. and Rigo, Michel},
     TITLE = {Numeration systems on a regular language},
   JOURNAL = {Theory Comput. Syst.},
  FJOURNAL = {Theory of Computing Systems},
    VOLUME = {34},
      YEAR = {2001},
    NUMBER = {1},
     PAGES = {27--44},
      ISSN = {1432-4350,1433-0490},
   MRCLASS = {68Q45},
  MRNUMBER = {1799066},
MRREVIEWER = {Fr\'ed\'erique\ Bassino},
       DOI = {10.1007/s002240010014},
       URL = {https://doi.org/10.1007/s002240010014},
}

@book{CANT10,
 Title = {{Combinatorics, Automata, and Number Theory}},
 Series = {Encyclopedia of Mathematics and its Applications},
 ISSN = {0953-4806},
 Volume = {135},
 Year = {2010},
 editor ={Val\'erie Berth\'e and Michel Rigo},
place={Cambridge},
 doi={10.1017/CBO9780511777653},
 publisher={Cambridge University Press, Cambridge},
}

@article {Bertrand-Mathis-1989,
    AUTHOR = {Bertrand-Mathis, Anne},
     TITLE = {Comment \'ecrire les nombres entiers dans une base qui n'est
              pas enti\`ere},
   JOURNAL = {Acta Math. Hungar.},
  FJOURNAL = {Acta Mathematica Hungarica},
    VOLUME = {54},
      YEAR = {1989},
    NUMBER = {3-4},
     PAGES = {237--241},
       DOI = {10.1007/BF01952053},
       URL = {https://doi.org/10.1007/BF01952053},
}

@incollection {Charlier-Cisternino-Stipulanti-2022,
    AUTHOR = {Charlier, \'Emilie and Cisternino, C\'elia and Stipulanti,
              Manon},
     TITLE = {A full characterization of {B}ertrand numeration systems},
 BOOKTITLE = {Developments in language theory},
    SERIES = {Lecture Notes in Comput. Sci.},
    VOLUME = {13257},
     PAGES = {102--114},
 PUBLISHER = {Springer, Cham},
      YEAR = {2022},
       DOI = {10.1007/978-3-031-05578-2\_8},
}

@book {Rigo-2014-2,
    AUTHOR = {Rigo, Michel},
     TITLE = {Formal languages, automata and numeration systems.  {V}ol.  2},
    SERIES = {Networks and Telecommunications Series},
      NOTE = {Applications to recognizability and decidability},
 PUBLISHER = {ISTE, London; John Wiley \& Sons, Inc., Hoboken, NJ},
      YEAR = {2014},
     PAGES = {xix+234},
      ISBN = {978-1-84821-788-1},
	   DOI = {10.1002/9781119042853}
}

@article {Fabre-1995,
    AUTHOR = {Fabre, St\'ephane},
     TITLE = {Substitutions et {$\beta$}-syst\`emes de num\'eration},
   JOURNAL = {Theoret. Comput. Sci.},
  FJOURNAL = {Theoretical Computer Science},
    VOLUME = {137},
      YEAR = {1995},
    NUMBER = {2},
     PAGES = {219--236},
       DOI = {10.1016/0304-3975(95)91132-A},
       URL = {https://doi.org/10.1016/0304-3975(95)91132-A},
}

@article {Gheeraert-Romana-Stipulanti-2024,
    AUTHOR = {Gheeraert, France and Romana, Giuseppe and Stipulanti, Manon},
     TITLE = {String attractors of some simple-{P}arry automatic sequences},
   JOURNAL = {Theory Comput. Syst. },
  FJOURNAL = {Theory of Computing System},
    VOLUME = {68},
      YEAR = {2024},
    NUMBER = {},
     PAGES = {1601-1621},
       DOI = {10.1007/s00224-024-10195-7},
       URL = {https://doi.org/10.1007/s00224-024-10195-7},
}

@incollection {Gheeraert-Romana-Stipulanti-2023,
    AUTHOR = {Gheeraert, France and Romana, Giuseppe and Stipulanti, Manon},
     TITLE = {String attractors of fixed points of {$k$}-{B}onacci-like
              morphisms},
 BOOKTITLE = {Combinatorics on words},
    SERIES = {Lecture Notes in Comput. Sci.},
    VOLUME = {13899},
     PAGES = {192--205},
 PUBLISHER = {Springer, Cham},
      YEAR = {2023},
       DOI = {10.1007/978-3-031-33180-0\_15},
}

@book {Walnut2,
    AUTHOR = {Shallit, Jeffrey},
     TITLE = {The logical approach to automatic sequences---exploring
              combinatorics on words with {\tt {W}alnut}},
    SERIES = {London Mathematical Society Lecture Note Series},
    VOLUME = {482},
 PUBLISHER = {Cambridge University Press, Cambridge},
      YEAR = {2023},
     PAGES = {xv+358},
      ISBN = {978-1-108-74524-6},
   MRCLASS = {68-01 (03B70 11B85 68R15 68W30)},
  MRNUMBER = {4589478},
MRREVIEWER = {Bj\o rn\ Kjos-Hanssen},
	   DOI = {10.1017/9781108775267}
}

@misc{Walnut1,
AUTHOR = {Mousavi, Hamoon},
TITLE  = {Automatic theorem proving in {W}alnut},
NOTE   = {Preprint available at \url{https://arxiv.org/abs/1603.06017}},
YEAR = {2016},
}

@article {Surer-2020,
    AUTHOR = {Surer, Paul},
     TITLE = {Substitutive number systems},
   JOURNAL = {Int. J. Number Theory},
  FJOURNAL = {International Journal of Number Theory},
    VOLUME = {16},
      YEAR = {2020},
    NUMBER = {8},
     PAGES = {1709--1751},
       DOI = {10.1142/S1793042120500906},
       URL = {https://doi.org/10.1142/S1793042120500906},
}

@article {Miro-Rust-Sadun-Tadeo-2023,
    AUTHOR = {Miro, Eden Delight and Rust, Dan and Sadun, Lorenzo and Tadeo,
              Gwendolyn},
     TITLE = {Topological mixing of random substitutions},
   JOURNAL = {Israel J. Math.},
  FJOURNAL = {Israel Journal of Mathematics},
    VOLUME = {255},
      YEAR = {2023},
    NUMBER = {1},
     PAGES = {123--153},
       DOI = {10.1007/s11856-022-2406-3},
       URL = {https://doi.org/10.1007/s11856-022-2406-3},
}

@article {Marshall-Maldonado-2024,
    AUTHOR = {Marshall-Maldonado, Juan},
     TITLE = {Lyapunov exponents of the spectral cocycle for topological
              factors of bijective substitutions on two letters},
   JOURNAL = {Discrete Contin. Dyn. Syst.},
  FJOURNAL = {Discrete and Continuous Dynamical Systems. Series A},
    VOLUME = {44},
      YEAR = {2024},
    NUMBER = {7},
     PAGES = {2068--2092},
       DOI = {10.3934/dcds.2024019},
       URL = {https://doi.org/10.3934/dcds.2024019},
}

@article{Renyi-1957,
	title = {Representations for real numbers and their ergodic properties},
	volume = {8},
	doi = {10.1007/BF02020331},
	language = {en},
	number = {3-4},
	urldate = {2022-11-14},
	journal = {Acta Math. Acad. Sci. Hung.},
	fjournal = {Acta Mathematica Academiae Scientiarum Hungaricae},
	author = {Rényi,  Alfréd},
	month = sep,
	year = {1957},
	pages = {477--493},
}

@article {Parry-1960,
    AUTHOR = {Parry, William},
     TITLE = {On the {$\beta $}-expansions of real numbers},
   JOURNAL = {Acta Math. Acad. Sci. Hung.},
  FJOURNAL = {Acta Mathematica. Academiae Scientiarum Hungaricae},
    VOLUME = {11},
      YEAR = {1960},
     PAGES = {401--416},
       DOI = {10.1007/BF02020954},
       URL = {https://doi.org/10.1007/BF02020954},
}

@incollection {Kreczman-Labbe-Stipulanti-2025,
    AUTHOR = {Kreczman, Savinien and Labb\'e, S\'ebastien and Stipulanti,
              Manon},
     TITLE = {A succinct study of positionality for {D}umont-{T}homas
              numeration systems},
 BOOKTITLE = {Combinatorics on words},
    SERIES = {Lecture Notes in Comput. Sci.},
    VOLUME = {15729},
     PAGES = {179--191},
 PUBLISHER = {Springer, Cham},
      YEAR = {2025},
       DOI = {10.1007/978-3-031-97548-6\_16},
}

@article{Surer-2018,
 author = {Surer, Paul},
 title = {Substitutions, coding prescriptions and integer representation},
 fjournal = {Journal of Number Theory},
 journal = {J. Number Theory},
 volume = {190},
 pages = {367--393},
 year = {2018},
 doi = {10.1016/j.jnt.2018.03.003},
}

\end{document}